\documentclass[11pt,reqno]{amsart}

\usepackage[margin=1.1in]{geometry}

\usepackage{amssymb,cite}
\usepackage[mathscr]{eucal}
\usepackage{eufrak}
\usepackage{xspace}
\usepackage{mathtools}
\usepackage{color}

\usepackage{hyperref}
\hypersetup{hidelinks}

\usepackage{amsfonts}
\usepackage{graphics}
\usepackage{indentfirst}
\usepackage{cite}
\usepackage{latexsym}
\usepackage{amsmath}
\usepackage{amssymb}
\usepackage[dvips]{epsfig}
\usepackage{amscd}
\usepackage{float}
\usepackage{graphicx}
\usepackage{subcaption}
\newtheorem{claim}{Claim}[section]

\newtheorem{theorem}{Theorem}[section]
\newtheorem{lemma}{Lemma}[section]
\newtheorem{proposition}{Proposition}[section]

\theoremstyle{definition}
\newtheorem{definition}{Definition}[section]

\theoremstyle{remark}
\newtheorem{remark}{Remark}[section]
  
\newcommand{\n}{\rho}
\newcommand{\ti}{\tilde}
\renewcommand{\div}{ {\rm div}}
\newcommand{\na}{\nabla }

\newcommand{\ga}{\gamma}

\newcommand{\de}{\delta}
\newcommand{\eps}{\epsilon}

\newcommand{\la}{\label}

\renewcommand{\div}{{\rm div}}

\numberwithin{equation}{section}

 \allowdisplaybreaks
 
\newcommand{\bnn}{\begin{eqnarray*}}
\newcommand{\enn}{\end{eqnarray*}}

\newcommand{\ba}{\begin{aligned}}
\newcommand{\ea}{\end{aligned}}
\newcommand{\be}{\begin{equation}}
\newcommand{\ee}{\end{equation}}

\def\p{\partial} 

\def\lap{\triangle}
 
\def\a{\alpha}
\def\om{\Omega}
\def\p{\partial}

\def\eps{\epsilon}

\def\a{\alpha}
\def\u{{\bf u}}

\def\r{\mathbb{R}^{3}}

\def\lap{\triangle}
\def\d{{\rm d}}
\def \n {\vec{n}}
\def \nu {\ti{{\bf u}}}

\makeatletter
\@addtoreset{equation}{section}
\makeatother 

\begin{document}

\title[Cahn-Hilliard-Navier-Stokes equations with singular potential]
{On weak solutions  for  the stationary Cahn-Hilliard-Navier-Stokes equations with singular potential}

\author[Z. Liang]{Zhilei Liang}
\address{School of  Mathematics, Southwestern University of Finance and Economics, Chengdu  611130,  China.}
\email{liangzl@swufe.edu.cn}

\author[S. Liu]{Sen Liu}
\address{School of  Mathematics, Southwestern University of Finance and Economics, Chengdu  611130,  China.}
\email{liu1835095756@163.com}

\author[J. Shuai]{Jiangyu Shuai}
\address{Faculty of Science, Civil Aviation Flight University of China,  Chengdu 618307, China.}
\email{shuaijy@cafuc.edu.cn}

\author[D. Wang]{Dehua Wang}
\address{Department of Mathematics, University of Pittsburgh, Pittsburgh, PA, 15260, USA.}
\email{dhwang@pitt.edu}
 

 
 \date{\today}
 
 \begin{abstract}

The stationary Navier--Stokes--Cahn--Hilliard equations are considered, governing the motion of a compressible, two-phase fluid mixture with a diffuse interface. The free energy density in this paper has a singular logarithmic (Flory–Huggins) form, ensuring that the mass fraction remains in the physical range and allowing for vacuum states. We prove the existence of weak solutions in a three-dimensional bounded domain under structural assumptions on the adiabatic exponent.
%
The stationary setting poses two main mathematical challenges: the absence of an energy inequality driven by the evolution process to control the singular potential, and the degeneracy of the density near the vacuum. To address these issues, we introduce a specialized regularization of the logarithmic term that eliminates the quadratic growth induced by anti-diffusion, thereby restoring compactness. 
Uniform estimates are obtained through a special choice of artificial pressure and an interpolation argument that controls the desired norm of the density. A two-level limiting process then yields a weak solution that satisfies the physical bounds almost everywhere on the support of the density. 
To our knowledge, this is the first existence result for the steady compressible Navier--Stokes--Cahn--Hilliard system that incorporates both a singular free energy and vacuum regions.

 \end{abstract}

\maketitle


\section{Introduction}
We are concerned with the stationary system of partial differential equations that describes the motion of a three-dimensional compressible two-phase fluid mixture with a diffuse interface:
\begin{eqnarray}\label{eq1.1}
\left\{\begin{aligned}
&{\rm div}(\rho\u)=0,\\
&{\rm div}(\rho\u\otimes\u)={\rm div}(\mathbb{S}_{ns}+\mathbb{S}_c-P\mathbb{I})+\rho {\mathbf g}_1+ \mathbf{g}_2,\\
&{\rm div}(\rho\u c)=\Delta\mu,\\
&\rho\mu=\rho\frac{\partial f(\rho,c)}{\partial c}-\Delta c,
\end{aligned}\right.
\end{eqnarray} 
where $\rho \ge 0$ denotes the total density of the binary fluid mixture, $c$ is the mass fraction of one component (with the physical constraint $|c| \le 1$), $\mathbf{u} = (u_1,u_2,u_3)$ is the mass-averaged velocity field, and $\mu$ is the chemical potential. The functions ${\textbf g}_1$ and $\mathbf{g}_2$ are the external body forces: the term $\rho \mathbf{g}_1$ acts on the bulk mass, while $\mathbf{g}_2$ acts on the fluid boundary. The viscous stress tensor $\mathbb{S}_{ns}$ and the capillary (Korteweg-type) stress tensor $\mathbb{S}_c$ are given by
\begin{eqnarray}\label{eq1.2}
\mathbb{S}_{ns}=\lambda_1(\na\u+(\na\u)^T)+\lambda_2 {\rm div}\u\mathbb{I}, \quad 
\mathbb{S}_c=-\na c\otimes\na c+\frac{1}{2}|\na c|^2\mathbb{I},
\end{eqnarray}
where $\mathbb{I}$ is the identity matrix. The viscosity coefficients $\lambda_1$ and $\lambda_2$ satisfy the physical constraints $\lambda_1 > 0$ and $2\lambda_1 + 3\lambda_2 \ge 0$. The Helmholtz free energy density $f(\rho,c)$ determines the total pressure $P$ via the thermodynamic relation $$P(\rho,c) = \rho^2\frac{\partial f(\rho,c)}{\partial \rho}.$$
 Typically, the free energy decomposes as
 \bnn f(\rho,c)=\int_{1}^{\rho}\frac{p(s)}{s^{2}}\d s+f_{mix}(\rho, c),\enn where $p$ is the inner pressure and $f_{mix}(\rho,c)$ is the mixing contribution. 
 In this work, we consider a physically relevant singular free energy of the form
  \be\label{eq1.4}\ba
  f(\rho,c)= \rho^{\gamma-1}+H(c) \ln\rho + F(c),
\ea\ee
where $\ga>1$ is the adiabatic exponent,  $H(c)$ is a smooth function and  $F(c)$ is a logarithmic (also known as Flory-Huggins type) potential of the form   
\be\la{eq1.4x}  F(c)=\frac{\theta_0}{2}\left((1+c)\ln(1+c)+(1-c)\ln(1-c)\right)-\frac{\theta_c}{2}c^2,\quad c\in (-1,1).\ee 
We note that the logarithmic term \eqref{eq1.4x} represents the mixture's entropy, and the temperature restriction $0<\theta_0<\theta_c$ ensures that a double-well form actually occurs. If $F$ is convex, the components are miscible and diffusion occurs; if $F$ is non-convex, there is a spinodal region in which anti-diffusion prevents the fluids from mixing completely.

We are interested in the boundary value problem of \eqref{eq1.1} in a three-dimensional bounded domain $\Omega\subset \mathbb{R}^3$. 
The system  \eqref{eq1.1} is closed by the total mass and relative mass constraints 
\begin{eqnarray}\label{eq1.8}
\int_{\om}\rho(x)\d x=m_1>0,\quad \int_{\om}\rho(x)c(x)\d x=m_2,
\end{eqnarray}
where $m_1, m_2$ are prescribed constants, and by the boundary conditions
\begin{eqnarray}\label{eq1.9}
\u=0,\quad\frac{\partial c}{\partial \n}=0,\quad \frac{\partial\mu}{\partial \n}=0\quad{\rm on}\,\, \partial\Omega,
\end{eqnarray}
where $\n=\n(x)$ is the outward unit normal  at   boundary point $x\in \p \om$.



The Cahn--Hilliard equation, when combined with the Navier--Stokes equations, serves as a fundamental diffuse-interface model for binary fluid flows, applicable from turbulent multi-phase dynamics to microfluidics. The dynamic version of this coupled system is known as ``Model H'' and was introduced by Hohenberg and Halperin \cite{hh} for density-matched incompressible fluids.  A comprehensive mathematical theory has been developed for this evolutionary system. Abels \cite{A,AA} proved the global existence of weak solutions. Giorgini, Miranville, and Temam \cite{GMT} addressed uniqueness, regularity, and strong solutions. Miranville and Temam \cite{MT} showed the existence for a Cahn--Hilliard--Navier--Stokes system based on the Oono model with singular potential. Further results incorporating Flory--Huggins potentials can be found in \cite{F,GGP,HKP}.
 
   
   A preliminary compressible variant was introduced by Lowengrub and Truskinovsky \cite{lt}, wherein the total density is a function of the concentration, $\rho = \rho(c)$, and the chemical potential assumes a more complex form that incorporates the pressure and the divergence of the velocity:
\bnn \mu=f_{0}'(c)-\frac{\rho'(c)}{\rho^{2}}\left(p-\big(\lambda_{1}+\frac{2\lambda_{2}}{3}\big)\div \u\right)-\frac{\div(\rho \na c)}{\rho},\enn
where $f_{0}(c)$ is the Helmholtz free energy.  
%
To handle the lack of compactness for $\nabla c$ in vacuum regions, Abels and Feireisl \cite{AF} introduced a modified compressible model:
\begin{eqnarray} \la{time}
\left\{\begin{aligned}
&\p_{t}\rho +{\rm div}(\rho\u)=0,\\
&\p_{t}(\rho \u)+{\rm div}(\rho\u\otimes\u)={\rm div}(\mathbb{S}_{ns}+\mathbb{S}_c-P\mathbb{I}),\\
&\p_{t}(\rho c)+{\rm div}(\rho\u c)=\Delta\mu,\\
&\rho\mu=\rho\frac{\partial f(\rho,c)}{\partial c}-\Delta c.
\end{aligned}\right.
\end{eqnarray}  
In \cite{AF}, Abels and Feireisl proved the existence of weak solutions to \eqref{time} under the assumption of a sufficiently regular potential. More recently, Basarić and Giorgini \cite{BG} extended this analysis to include singular free energies.

 The system  \eqref{eq1.1} studied here is the stationary counterpart of \eqref{time}. When the external forces vanish ($\mathbf{g}_1 = \mathbf{g}_2 = 0$), its solutions align with the critical points of the total energy functional 
\be\la{time1} \mathcal{E}(t)=\int_{\om}\left(\frac{1}{2}\rho |\u|^{2}+\rho f(\rho,c)+\frac{1}{2}|\na c|^{2}\right)(x,t) \d x,\quad t\ge0.\ee
 The stationary problem differs fundamentally from the evolutionary one; it is not just a parallel case or a broader generalization. For the time-dependent equations \eqref{time}, the {\it a  priori} energy inequality $\mathcal{E}(t) \le \mathcal{E}(0)$ provides a crucial tool for the mathematical analysis. In the stationary setting, no such time-evolution bound is available, and the energy balance by itself yields far less information. When the mixing free energy $f_{\text{mix}}(\rho,c)$ is regular, namely,   the functions in \eqref{eq1.4} satisfy 
 \be\la{lw} |H(c)|+|H'(c)|+ |F(c)|+|F'(c)|\le C,\quad c\in \mathbb{R},\ee
 Liang-Wang\cite{LW,LW2} proved the existence of weak solutions for equations \eqref{eq1.1}  if the adiabatic exponent $\ga$  satisfies  $\ga>4/3.$   For stationary Cahn--Hilliard--Navier--Stokes equations of incompressible flows, we refer to  \cite{AW,BDMM,BDM,KPS,KPS2}.
%
 For the steady compressible Navier--Stokes equations (without phase separation), Lions \cite{Lions} established the existence of a nontrivial equilibrium state for $\gamma \ge 5/3$. Subsequently, the admissible range of $\gamma$ was extended to $\gamma > 1$ in various settings; see \cite{FNP,FSW,PW}  for bounded domains with Dirichlet boundary conditions and \cite{BB,JZ} for periodic domains. Comprehensive treatments of steady compressible fluid flows can be found in the monographs  \cite{MPZ,NS}.

However, the regular mixture entropy violates the physical constraint $c \in [-1,1]$. A physically consistent Helmholtz free energy---particularly one with a singular logarithmic potential---is therefore both more relevant for applications and more challenging mathematically. In this paper, we assume a singular free energy and establish the existence of weak solutions to the stationary problem \eqref{eq1.1}-\eqref{eq1.9}.
%
\begin{definition}[Weak solution]\la{de} Let  exponents $p>\frac{6}{5}$ and   
$\theta>0$.  A quadruple  $(\rho,\u,\mu,c)$ is a weak solution to problem   \eqref{eq1.1}-\eqref{eq1.9}, if 
\bnn \ba 0\le \rho\in L^{\gamma+\theta}(\Omega),\quad   \u\in H_0^1(\Omega;\r),\quad (\mu, c)\in H^1(\Omega)\times W^{2,p}(\Omega),\ea\enn
and  the  following hold:
\begin{enumerate}\addtolength{\itemsep}{+0.8\baselineskip}
\item  Equation $\eqref{eq1.1}_1$  is  satisfied in  distributional  sense.  
Extended by zero outside $\Omega$, the pair $(\rho,\u)$ satisfies the renormalized continuity equation
\bnn {\rm div}(b(\rho)\u)+(b'(\rho)\rho-b(\rho)){\rm div}\u=0\quad {\rm in}\,\,\,\mathcal{D}'(\om),\enn
for any $b\in C^1(\mathbb{R})$ with $b'(z)=0$ for large $z$.
 \item 
 The following weak formulation of momentum and concentration equations holds:
%
$$
 \int_{\Omega}\Bigl(\rho \mathbf{u}\otimes \mathbf{u} + \rho^{2}\frac{\partial f(\rho,c)}{\partial\rho}\mathbb{I} - \mathbb{S}_{ns} - \mathbb{S}_{c}\Bigr):\nabla \Phi\,\mathrm{d}x = \int_{\Omega}(\rho \mathbf{g}_1 + \mathbf{g}_2)\cdot \Phi\,\mathrm{d}x,$$
 for any $\Phi\in C_0^\infty(\Omega;\mathbb{R}^3)$; and
$$\int_{\Omega}(\rho c\mathbf{u} - \nabla \mu)\cdot \nabla \phi\,\mathrm{d}x = 0,\quad
\int_{\Omega}\Bigl(\rho \mu - \rho\frac{\partial f(\rho,c)}{\partial c}\Bigr)\phi\,\mathrm{d}x = \int_{\Omega}\nabla c\cdot \nabla \phi\,\mathrm{d}x,$$
for any $\phi\in C^\infty(\overline{\Omega})$.

\item 
The mass constraints \eqref{eq1.8} are satisfied, and the following energy inequality holds:
\be\la{ls40}\int_{\om}(\lambda_1|\na\u|^2+(\lambda_1+\lambda_2)({\rm div}\u)^2+|\na\mu|^2)\d x\leq \int_{\om}(\rho  \mathbf{g}_1+ \mathbf{g}_2)\cdot\u\d x.\ee
\end{enumerate}
\end{definition}

Our main result is the following existence theorem.

\begin{theorem}\label{t} 
Let $\Omega\subset\mathbb{R}^3$ be a bounded, simply connected domain with $C^2$ boundary. Assume the external forces satisfy
\begin{eqnarray}\label{eq1.11g} 
{\mathbf{g}_1},\,\,{\mathbf{g}_2}\in L^\infty\left(\Omega;\mathbb{R}^3\right)
\end{eqnarray}
 and that $H(c)$ in \eqref{eq1.4} is  constant, i.e., 
\begin{eqnarray}\label{eq1.12}
  H(c)=H=\mathrm{const}>0.
\end{eqnarray} 
Suppose the adiabatic exponent $\gamma$ satisfies
\begin{eqnarray}\label{eq1.11} 
\gamma>\frac{3}{2}\,\,\,\,\, {\rm if}\,\,\,\na\times {\mathbf{g}_1}=0,\quad\quad \gamma>\frac{5}{3}\,\,\,\,\,  {\rm if}\,\, \na\times {\mathbf{g}_1}\neq 0,
\end{eqnarray}
 and the total and relative masses satisfy
\begin{eqnarray}\label{aver}
m_1>0, \quad m_2\in (-m_{1},\,m_{1}).
\end{eqnarray} Then,  the problem \eqref{eq1.1}-\eqref{eq1.9} admits a weak solution $(\rho,\u,\mu,c)$ in the sense of Definition 1.1. Moreover, the concentration satisfies the physical bounds
\bnn-1<c(x)<1\quad a.e.\,\,\,{\rm in}\,\,\,\,\{x\in \om:\,\rho(x)>0\}\enn
and the singular term is well-defined:
\be\la{kk}
\rho\frac{\partial f(\rho,c)}{\partial c}=
\begin{cases}
\rho F'(c), & \text{if } \rho > 0, \\
0, & \text{if } \rho = 0.
\end{cases}
\ee
\end{theorem}

\begin{remark} 
The mixing free energy density is motivated by the logarithmic Helmholtz energy typical for Cahn--Hilliard models (cf. \cite{AF,BG}):
\bnn\ba 
  f_{mix}(\rho,c) &=\frac{\theta_{0}}{2} \left( (1+c)\ln(\rho(1+c))+ (1-c)\ln(\rho(1-c))\right)-\frac{\theta_{c}}{2} c^{2}\\
  &=H(c)\,\ln \rho +F(c)\\
  &=\underbrace{\theta_{0}}_{H(c)}\,\ln \rho +\underbrace{\frac{\theta_0}{2}\left((1+c)\ln(1+c)+(1-c)\ln(1-c)\right)-\frac{\theta_c}{2}c^2}_{F(c)},\quad c\in (-1,1).
\ea\enn   
More generally, we may allow $F$ to satisfy, for some $\alpha_0$,
  \bnn \lim_{c\rightarrow\pm1}F'(c)=\pm\infty,\quad \lim_{c\rightarrow\pm1}F''(c)=+\infty,\quad F''(c)\ge \a_{0},\enn
and assume $H$ is bounded for some positive constants $\underline{H}$ and $\overline{H}$, 
\begin{eqnarray}\label{eq1.12x}
0\le \underline{H}< H(c),\,H'(c)\leq \overline{H}<\infty.
\end{eqnarray} 
\end{remark}

\begin{remark} Based on results for the single-fluid steady compressible Navier--Stokes equations  \cite{FSW,MPZ,PW},   one expects that the range of adiabatic exponent $\ga$ in \eqref{eq1.11} can be further extended.   However,   it seems difficult  to deal with  the  weighted estimate on $c$ near the boundary, because the singularity of the free energy only yields control on $\|\sqrt{\rho} \frac{\p f(\rho,c)}{\p c}\|_{L^{2}}$.
\end{remark}

\begin{remark} 
To the best of our knowledge, this is the first existence result for weak solutions of the stationary compressible Navier--Stokes--Cahn--Hilliard system that simultaneously allows for a singular free energy potential and the presence of vacuum states.
\end{remark}

{\bf Outline of the proof and key ideas.}
We establish Theorem \ref{t}   via a two-fold weak convergence procedure applied to a carefully constructed approximate system. The main difficulties arise from the singular nature of the free energy and the possible degeneracy of the density. We outline our strategy below.

%
 For the evolutionary problem, the {\it a priori} energy inequality $\mathcal{E}(t) \le \mathcal{E}(0)$ provides a uniform $L^2$ bound on $c$, making Taylor expansion regularization of the singular potential feasible (see, e.g., \cite{GT,BG,CFG}). This approach fails for the stationary problem because \eqref{eq1.1} alone yields no such control, and the $L^1$ bounds from \eqref{eq1.8} are insufficient to control the nonlinearity in the potential. Our idea is to introduce a special approximate function $f_2^\delta(c)$ for the singular logarithmic term, constructed so that the difference $f_2^{\de}(c)-\frac{\theta_{c}}{2}c^{2}$ increases only linearly as $|c|$ becomes large. This cancellation effectively offsets the quadratic growth of the anti-diffusion term. See Remark \ref{r2.2} for further details.
   
We first introduce and analyze the regularized approximate system \eqref{eq2.7}-\eqref{eq2.8}, and prove the existence of solutions via a fixed point argument, following  \cite[Theorem 4.1]{LW}. The proof, however, is not a straightforward extension: one must carefully pair the function $f_2^\delta$ with the quadratic term $\frac{\theta_c}{2}c^2$ to ensure they cancel each other.
 
Then we need to establish the estimates uniform in $\delta$  for the approximate solutions (Proposition~4.1). A major challenge is controlling the $L^{\frac{3}{2}}$ integrability of the density. Our strategy  is  to choose a specific artificial pressure, and use the observation
\[
\frac{\partial f_2^\delta}{\partial c} - \theta_c c = O\Bigl(\ln \frac{1}{\delta}\Bigr) \quad \text{for } |c| > 1+\delta,
\]
and interpolation to obtain an estimate of the form
\[
\|\rho\|_{L^{3/2}(\{|c|\ge 1+\delta\})} \le \frac{1}{2}\|\rho (f_2^{\delta'} - \theta_c c)c\|_{L^1} + \frac{1}{2}\Bigl(\ln\frac{1}{\delta}\Bigr)^{-1}\|\rho\|_{L^{11}}^{11} + C,
\]
which allows us to close the estimates for all $\gamma > 3/2$. 
Another important observation is that condition \eqref{aver}    is also essential for obtaining $L^p$ bounds on $\mu$.

%
 Finally, when passing to the limit as the regularization parameters vanish, we must ensure that the limiting concentration $c$ remains strictly within the physical interval $(-1,1)$ on non-vacuum regions. The coupled Navier--Stokes--Cahn--Hilliard structure makes this delicate, especially since the pressure depends jointly on $\rho$ and $c$ and lacks monotonicity in $\rho$ for all $c$. Adapting techniques from \cite{Lions, FNP,MPZ,NS} and \cite{LW,LW2}, we prove strong convergence of the density. Then, following ideas from \cite{BG}, we establish that $c\in(-1,1)$ almost everywhere where $\rho > 0$, thereby verifying the physical constraint and ensuring the convergence of $F^\delta(c^\delta)$ and the well-definedness of $\partial_c f(\rho,c)$ in the non-vacuum domains.
  
 The rest of the paper is organized as follows. 
In Section~2, we construct a carefully designed regularization of the singular logarithmic potential. This approximation is essential for the stationary problem, as it offsets the quadratic growth of the anti-diffusion term and restores the necessary compactness. 
Section~3 focuses on the first-level approximate system with parameters $\epsilon$ and $\delta$. We prove the existence of strong solutions using a fixed-point argument and establish uniform-in-$\epsilon$ estimates. Taking the limit $\epsilon \to 0$, we obtain a weak solution to an intermediate system with the regularized free energy.
In Section~4, we derive uniform-in-$\delta$ estimates for the solutions of this intermediate system. A crucial step is to control the $L^{3/2}$-norm of the density via a delicate interpolation argument that exploits the specific structure of the regularized potential. We then take the $\delta \to 0$ limit to recover the original singular free energy.
Section~5 contains the proof of our main result, Theorem~1.1. We establish the strong convergence of the density and show that the limiting concentration $c$ satisfies the physical bounds $-1 < c < 1$ almost everywhere on the support of $\rho$. 
Finally, technical lemmas and auxiliary estimates are collected in the Appendix.

\textbf{Notation:} 
 %
For two vectors $\mathbf{a},\mathbf{b}\in\mathbb{R}^{3}$, we write $\mathbf{a}\otimes \mathbf{b} = (a_{i}b_{j})_{3\times 3}$. 
Let $C_0^\infty(\Omega;\mathbb{R}^3)$ be the set of all smooth and compactly supported vector-valued functions from $\Omega$ to $\mathbb{R}^3$, and $C_0^\infty(\Omega) = C_0^\infty(\Omega;\mathbb{R})$. 
$C^\infty(\overline{\Omega})$ denotes the set of uniformly smooth functions on $\overline{\Omega}$. 
For any $p\in[1,\infty]$ and $k\in\mathbb{N}$, we use $W^{k,p}(\Omega;\mathbb{R}^3)$ and $W^{k,p}(\Omega)$ to denote the Sobolev spaces of functions valued in $\mathbb{R}^3$ and $\mathbb{R}$, respectively. 
We write $H^k = W^{k,2}$, $L^p = W^{0,p}$, and let $W_0^{k,p}$ and $W_n^{k,p}$ indicate the subspaces with zero Dirichlet and zero Neumann boundary conditions, respectively. 
$|\Omega|$ denotes the Lebesgue measure of $\Omega$, and the average of a function is $(f)_\Omega = \frac{1}{|\Omega|}\int_{\Omega}f\,\mathrm{d}x$. 
We also denote by $\int f = \int_{\Omega}f(x)\,\mathrm{d}x$ and by $\overline{f}$ the weak limit of a sequence $(f^\delta)$. 
The characteristic function of a set $A\subset \Omega$ is $\mathbf{1}_A$. 
For asymptotic notation, we write $f = O(g)$ as $x\to x_0$ if there exists a constant $C$ such that $|f(x)|\le C|g(x)|$ for all $x$ sufficiently close to $x_0$, and $f = o(g)$ as $x\to x_0$ if $\lim_{x\to x_0} f(x)/g(x) = 0$.

\bigskip
  
\section{Regularization  and  Approximate System}  

As mentioned earlier, the usual Taylor expansion fails for the stationary problem \eqref{eq1.1} because there is no {\it a priori} $L^2$ bound on $c$ to control the quadratic nonlinearity $\frac{\theta_{c}}{2}c^{2}$. So we need a new approximation for the singular potential. 
From   \eqref{eq1.4},   we write 
\bnn F(c)= f_2(c)-\frac{\theta_c}{2}c^2,\enn
where $f_2(c)$ is a  singular function of  the form 
\be\la{f} f_2(c)=\frac{\theta_0}{2}\left((1+c)\ln(1+c)+(1-c)\ln(1-c)\right)\ee
that is approximated by the following   regularized extension:
 \begin{eqnarray*}
f_{2}^{\delta}(c)=\left\{\begin{aligned}
&f_{2}(c),&& 0\leq c\leq 1-\delta,\\[4pt]
&\frac{\theta_0}{2\delta(2-\delta)}(c-1+\delta)^2+\frac{\theta_0}{2}\ln\left(\frac{2}{\delta}-1\right)(c-1+\delta)+f_2(1-\delta),&& 1-\delta< c\leq 1,\\[4pt]
&\frac{\theta_c\delta(2-\delta)-\theta_0}{6\delta^2(2-\delta)}(c-1)^3+\frac{\theta_0}{2\delta(2-\delta)}(c-1)^2\\
&\quad +\left(\frac{\theta_0}{2-\delta}+\frac{\theta_0}{2}\ln\left(\frac{2-\delta}{\delta}\right)\right)(c-1)+\frac{\delta\theta_0}{2(2-\delta)}+\theta_0\ln(2-\delta),&&  1< c\leq 1+\delta,\\[4pt]
&\frac{\theta_c}{2}(c-1-\delta)^2+\left(\frac{\theta_c\delta}{2}+\frac{3\theta_0}{2(2-\delta)}+\frac{\theta_0}{2}\ln\left(\frac{2-\delta}{\delta}\right)\right)(c-1-\delta)&&\\
&\quad +\frac{\theta_c\delta^2}{6}+\frac{11\theta_0\delta}{6(2-\delta)}+\frac{\theta_0\delta}{2}\ln\left(\frac{2-\delta}{\delta}\right)+\theta_0\ln(2-\delta),&&1+\delta< c.
\end{aligned}\right.
\end{eqnarray*}
Additionally, we assume  
$$f_{2}^{\delta}(c)=f_{2}^{\delta}(-c)\quad \forall \,c\in (-\infty,\infty),$$ 
so that   $f_{2}^{\delta}(c)$ is well defined  when $c\le 0$.
\medskip 

\begin{remark} The function  $f_{2}^{\delta}(c)$  is $C^{2}$  over  $(-\infty,
\infty).$ We can easily compute its first-order  derivative 
\begin{eqnarray*}
f_{2}^{\delta '}(c)=\left\{\begin{aligned}
&\frac{\theta_0}{2}\ln\left(\frac{1+c}{1-c}\right), &&0 <c\leq 1-\delta,\\
&\frac{\theta_0}{\delta(2-\delta)}(c-1+\delta)+\frac{\theta_0}{2}\ln\left(\frac{2-\delta}{\delta}\right),&& 1-\delta<c\leq 1,\\
&\frac{\theta_c\delta(2-\delta)-\theta_0}{2\delta^2(2-\delta)}(c-1)^2+\frac{\theta_0}{\delta(2-\delta)}(c-1)+\frac{\theta_0}{2-\delta}+\frac{\theta_0}{2}\ln\left(\frac{2-\delta}{\delta}\right),&& 1<c\leq 1+\delta,\\
&\theta_c(c-1-\delta)+\frac{\theta_c\delta}{2}+\frac{3\theta_0}{2(2-\delta)}+\frac{\theta_0}{2}\ln\left(\frac{2-\delta}{\delta}\right),&& 1+\delta<c,
\end{aligned}\right.
\end{eqnarray*} 
 and  second-order  derivative  
\bnn
f_{2}^{\delta ''}(c)=\left\{\begin{aligned}
&\frac{\theta_0}{1-c^2},&& 0< c\leq 1-\delta,\\
&\frac{\theta_0}{\delta(2-\delta)},&& 1-\delta<c\leq 1,\\
&\frac{\theta_c\delta(2-\delta)-\theta_0}{\delta^2(2-\delta)}(c-1)+\frac{\theta_0}{\delta(2-\delta)},&& 1< c\leq 1+\delta,\\
&\theta_c,&&  1+\delta<c.
\end{aligned}\right.
\enn
Furthermore,  we have the following claims:
\begin{itemize}
\item  $f_{2}^{\delta}(c)$ converges uniformly to $f_2(c)$ in any compact set of $(-1,1)$.

\item   $f_{2}^{\delta '}(c)$ converges uniformly to $f_2'(c)$ in any compact set of $(-1,1)$.
\end{itemize}
Indeed, for any $-1<a<b<1$,   we  may  choose   $\delta>0$ so small to satisfy 
\[\delta<\min\{1+a,1-b\},\]
then,   $$f_{2}^{\delta}(c)=f_{2}(c),\quad \forall c\in [a,b].$$
Similar argument runs for     $f_{2}^{\delta '}(c)$. 
\end{remark}

\begin{remark}\la{r2.2}The construction of    $f_{2}^{\delta}(c)$  is  to offset the nonlinear term $\frac{\theta_{c}}{2}c^{2}$ appeared in  \eqref{eq1.4}.    In particular,  
\begin{itemize}
\item The function $f_{2}^{\delta}(c)-\frac{\theta_c}{2} c^2$  is  linear  with respect to  $c$, if $|c|$ is properly large.
\item  Let 
\begin{eqnarray}\label{eq2.3}
c^*=1-\frac{2}{1+\exp\{2\theta_c/\theta_0\}}.
\end{eqnarray}
Then there is a constant $M=M(\theta_{0},\theta_{c})$, independent of $\de$, such that  
\be\la{dan}\left|f_2^{\delta'}-\theta_cc\right|<M\,\,\,\,{\rm if}\,\, |c|\leq c^*, \quad\,\,\,\,(f_2^{\delta'}-\theta_cc)c>0\,\,\,\,{\rm if}\,\, |c|>c^*.\ee

\item  The function $f_{2}^{\delta ''}(c)-\theta_{c} $  is zero when  $|c|>1+\de$. If $|c|<1+\delta$,  
  there is a constant $C$  independent of $\delta$  such that
\be\label{ac**}
|f_2^{\delta''}-\theta_c|<C\,\,\, {\rm if}\,\,|c|\leq \sqrt{1-\frac{\theta_0}{\theta_c}},\quad \quad f_2^{\delta''}-\theta_c\geq 0\,\,\,{\rm if}\,\,\sqrt{1-\frac{\theta_0}{\theta_c}}\le  |c|\le 1+\de.\ee
\end{itemize}
The following graphs  illustrate the  functions  $f_{2}^{\delta}$  and $f_{2}^{\de}-\frac{\theta_{c}}{2}c^{2}$,    and their  derivatives ($\delta=0.1$, $\theta_0=1$, $\theta_c=1.5$).
\begin{figure}[ht]
    \centering
    \begin{subfigure}[b]{0.32\textwidth}
        \centering
        \includegraphics[width=\textwidth]{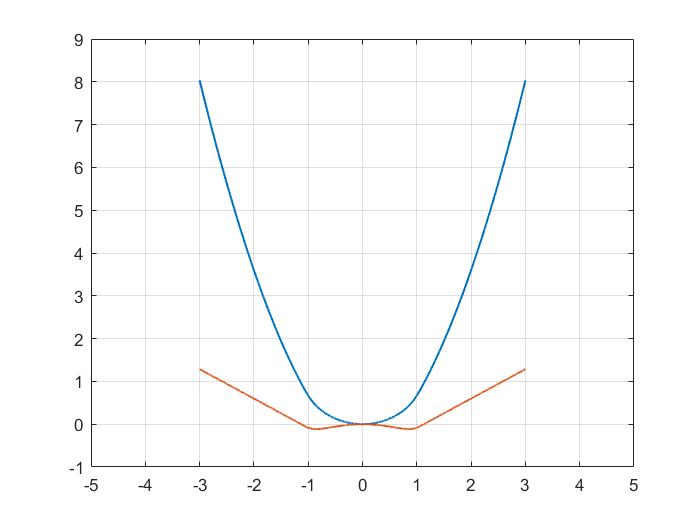}
        \caption{}
    \end{subfigure}
    \hfill
    \begin{subfigure}[b]{0.32\textwidth}
        \centering
        \includegraphics[width=\textwidth]{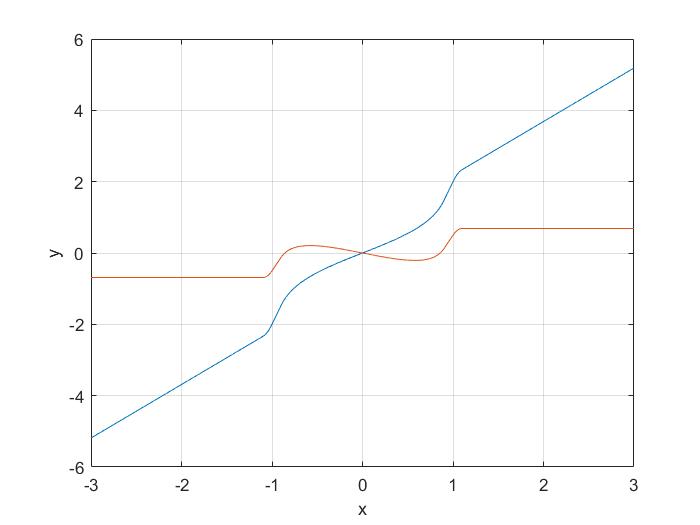}
        \caption{}
    \end{subfigure}
    \hfill
    \begin{subfigure}[b]{0.32\textwidth}
        \centering
        \includegraphics[width=\textwidth]{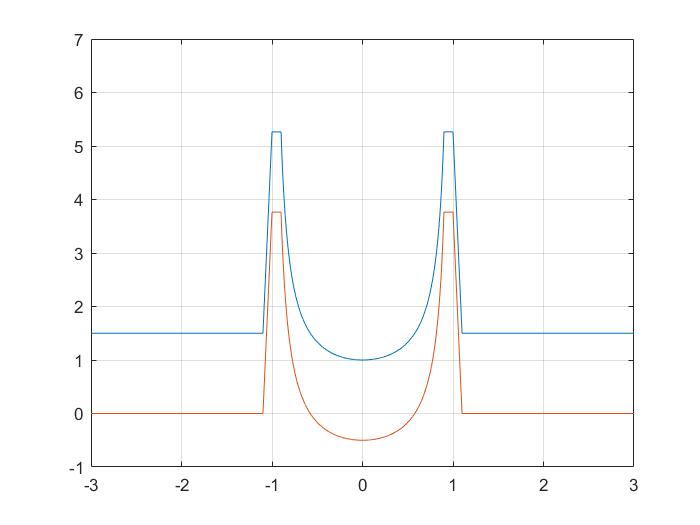}
        \caption{}
    \end{subfigure}
    \caption{The plot of $f_{2}^{\delta}$ (blue line) and $f_{2}^{\delta}-\frac{\theta_c}{2} c^2$ (red line) in (A), and their derivatives in (B) and (C)}
\end{figure}
\end{remark}

With above preparation, we  approximate the functions  in \eqref{eq1.4}-\eqref{eq1.4x}   by  
\be\label{eq1.4a}\ba
f^{\de}(\rho,c)&:=\rho^{\gamma-1}+H(c)\ln\rho +f_{2}^{\de}(c)-\frac{\theta_c}{2}c^2,\quad c\in (-\infty,\infty).
\ea\ee 
It is clear that, for fixed $\de>0$,  the function $f^{\de}(\rho,c)$ is regular with respect to $c$. 

Let  $m_1$ and $m_2$ be as in   \eqref{eq1.8}  and $|\Omega|$ be the Lebesgue measure of $\Omega$,  we define  \begin{eqnarray}\label{eq2.6}
\epsilon\in(0,1),\quad\delta\in(0,1),\quad\rho_0=\frac{m_1}{|\Omega|},\quad \rho _{0} c_0=\frac{m_2}{|\Omega|},
\end{eqnarray} 
and  consider the   approximation system:
\begin{eqnarray}\label{eq2.7}
\left\{\begin{aligned}
&\epsilon^2\rho+{\rm div}(\rho\u)=\epsilon^4\Delta\rho+\epsilon^2\rho_0,\\
&\epsilon^2\rho\u+{\rm div}(\rho\u\otimes\u)+\na\left(\left(\ln\frac{1}{\delta}\right)^{-1}\rho^{11}+\rho^2\frac{\partial f^\delta(\rho,c)}{\partial \rho}\right)+\epsilon^4\na\rho\cdot\na\u\\
&\quad={\rm div}\mathbb{S}_{ns}+\rho\mu\na c-\rho\frac{\partial f^\delta(\rho,c)}{\partial c}\na c+\rho \mathbf{g}_1+\mathbf{g}_2,\\
&\epsilon\rho c+\rho\u\cdot\na c=\Delta\mu+\epsilon\rho_0c_0,\\
&\rho\mu=\rho\frac{\partial f^\delta(\rho,c)}{\partial c}-\Delta c,\\
&\int\rho c=m_2+\epsilon\int(\rho_0-\rho)c-\epsilon^3\int\na\rho\cdot\na c,\\
& \int\rho\mu=\int\rho\frac{\partial f^\delta(\rho,c)}{\partial c}.
\end{aligned}\right.
\end{eqnarray}
We close 
equations \eqref{eq2.7}   with  boundary conditions
\begin{eqnarray}\label{eq2.8}
\u=0,\quad\frac{\partial\rho}{\partial \n}=0,\quad \frac{\partial c}{\partial \n}=0,\quad \frac{\partial\mu}{\partial \n}=0\quad {\rm on}\quad \partial\Omega.
\end{eqnarray}
 \begin{remark}
Due to the  Neumann boundary conditions in \eqref{eq2.8},  we impose  integral identity  $\eqref{eq2.7}_{5}$ to guarantee that $c$ is uniquely determined, and therefrom, impose   $\eqref{eq2.7}_6$   for the uniqueness of $\mu$.   Additionally,   in momentum equations $\eqref{eq2.7}_{2}$   we used the   formal  computation 
\[\rho\mu\na c-\rho\frac{\partial f^\delta(\rho,c)}{\partial c}\na c=-\Delta c\na c={\rm div}\mathbb{S}_{c}.\]  
\end{remark}

\bigskip

\section{Approximate System \eqref{eq2.7}-\eqref{eq2.8}}

In this section, we consider the approximate system \eqref{eq2.7}-\eqref{eq2.8} with parameters $\epsilon$ and $\delta$, prove the existence of strong solutions, establish uniform-in-$\epsilon$ estimates, and then take the limit $\epsilon \to 0$.
 
\subsection{Existence of solutions} 
We firsr prove the following existence result:
\begin{theorem}\label{thm2.1}
Let  conditions \eqref{eq1.12}-\eqref{eq1.11} and \eqref{eq2.6}  hold true.
Then,  for  $p\in(1,+\infty)$ and small $\eps>0$, the    problem \eqref{eq2.7}-\eqref{eq2.8} admits a solution   $(\rho^\epsilon,\u^\epsilon,\mu^\epsilon,c^\epsilon)$  which satisfies 
\begin{eqnarray}\label{eq2.13}
0\leq\rho^\epsilon\in W^{2,p}(\Omega),\quad \Vert\rho^\epsilon\Vert_{L^1(\Omega)}=m_1,
\end{eqnarray}
\begin{eqnarray}\label{eq2.14}
\u^\epsilon\in W_0^{1,p}(\Omega;\mathbb{R}^3)\cap W^{2,p}(\Omega;\mathbb{R}^3),\quad (\mu^\epsilon,c^\epsilon)\in W^{2,p}(\Omega)\times W^{2,p}(\Omega).
\end{eqnarray}
\end{theorem}
\begin{proof}  The proof of  Theorem \ref{thm2.1} is based on  the  Fixed Point Theorem.
From \cite[Proposition A.2.]{NS}, for     $\nu\in W_0^{1,\infty}(\Omega;\mathbb{R}^3)$,  the  equation $\eqref{eq2.7}_1$  with boundary condition $\frac{\p \rho}{\p \n}=0$  has a unique solution   $\rho=\rho(\nu)$ such that \begin{equation}\label{eq2.11}
\rho\geq 0,\quad a.e.\,\, {\rm in}\,\,\Omega,\quad \Vert\rho\Vert_{L^1}=m_1,\quad \Vert\rho\Vert_{W^{2,p}}\leq C(\epsilon,p,\Vert\nu\Vert_{W_{0}^{1,\infty}}).
\end{equation}
For $p\in (1,\infty)$, we set 
\begin{eqnarray}\label{eq2.15}
(\nu,\tilde{\mu},\tilde{c})\in\mathcal{W}:= W_0^{1,\infty}(\Omega;\mathbb{R}^3)\times W_{n}^{1,p}(\Omega)\times W_{n}^{1,p}(\Omega).\end{eqnarray}
Consider the elliptic system for $(\u,\mu,c)$:
\begin{eqnarray}\label{eq2.16}
\left\{\begin{aligned}
&{\rm div}\mathbb{S}_{ns}=F^1(\nu,\tilde{\mu},\tilde{c})\\
&~~:=\epsilon^2\rho\nu+{\rm div}(\rho\nu\otimes\nu)+\na\left(\left(\ln\frac{1}{\delta}\right)^{-1}\rho^{11}+\rho^2\frac{\partial f^\delta(\rho,\tilde{c})}{\partial\rho}\right)+\epsilon^4\na\rho\cdot\na\nu\\
&\quad\,\,\,\,+\rho\frac{\partial f^\delta(\rho,\tilde{c})}{\partial\tilde{c}}\na\tilde{c}-\rho\tilde{\mu}\na\tilde{c}-\rho \mathbf{g}_1-\mathbf{g}_2,\\
&\lap\mu= F^2(\nu,\tilde{\mu},\tilde{c}):=\epsilon\rho\tilde{c}+\rho\nu\cdot\na \tilde{c}-\epsilon\rho_0c_0,\\
&\lap c= F^3(\nu,\tilde{\mu},\tilde{c}):=\rho\frac{\partial f^\delta(\rho,\tilde{c})}{\partial \tilde{c}}-\rho\tilde{\mu},\\
&\int\rho\tilde{c}=m_2+\epsilon\int(\rho_0-\rho)\tilde{c}-\epsilon^3\int\na\rho\cdot\na \tilde{c},\quad \int\rho\tilde{\mu}=\int\rho\frac{\partial f^\delta(\rho,\tilde{c})}{\partial \tilde{c}},\\
&\u=0,\quad\frac{\partial\mu}{\partial \n}=0,\quad \frac{\partial c}{\partial \n}=0\,\,{\rm on}\,\, \partial\Omega.
\end{aligned}\right.
\end{eqnarray} 
According to \cite{GT}, system \eqref{eq2.16} admits a unique solution:
\begin{eqnarray}\label{eq2.17}
(\u,\mu,c):=\mathbb{A}[(\nu,\tilde{\mu},\tilde{c})].
\end{eqnarray}
We see that the operator $\mathbb{A}$ is compact and continuous in $\mathcal{W}$.  Moreover, we have 
\begin{claim}\la{ccc} Let  $\mathbb{A}:\mathcal{W}\to \mathcal{W}$ be as  in  \eqref{eq2.17} and  $\mathcal{W}$ be as  in \eqref{eq2.15}. Then the set of possible fixed points
\begin{eqnarray}\label{eq2.18}
\left\{(\u,\mu,c)\in \mathcal{W}\left|\begin{aligned}
&(\u,\mu,c):=\sigma \mathbb{A}[(\u,\mu,c)],\\
& for\,\, some\,\, \sigma\in(0,1]\,\, and\,\,\rho=\rho(\u)
\end{aligned}\right.\right\}
\end{eqnarray}
is bounded.
\end{claim}
If we temporarily accept Claim \ref{ccc}, then by Schaefer's Fixed Point Theorem ({\it cf.}\cite[Chapter 9, Theorem 3]{Evans}), we conclude 
$(\u,\mu,c):=\mathbb{A}[(\u,\mu,c)]$  with $\rho=\rho(\u).$
This establishes the existence of a solution $(\rho^\epsilon,\u^\epsilon,\mu^\epsilon,c^\epsilon)$ to \eqref{eq2.7} and \eqref{eq2.8}.

{\it Proof of Claim \ref{ccc}.} It suffices to show  there is a constant $M<\infty$ which is uniform in $\sigma$,  such that
\begin{eqnarray}\label{eq2.19}
\Vert(\u,\mu,c)\Vert_{\mathcal{W}}<M,
\end{eqnarray}
where    $(\rho,\u,\mu,c)$ solves
\begin{eqnarray}\label{eq2.20}
\left\{\begin{aligned}
&\epsilon^2\rho+{\rm div}(\rho\u)=\epsilon^4\Delta\rho+\epsilon^2\rho_0,\\
&{\rm div}\mathbb{S}_{ns}=\sigma F^1(\u,\mu,c),\\
&\Delta\mu=\sigma F^2(\u,\mu,c),\\
&\Delta c=\sigma F^3(\u,\mu,c),\\
&\int\rho c=m_2+\epsilon\int(\rho_0-\rho)c-\epsilon^3\int\na\rho\cdot\na c,\quad \int\rho\mu=\int\rho\frac{\partial f^\delta(\rho,c)}{\partial c},\\
&\u=0,\quad\frac{\partial\rho}{\partial \n}=0,\quad\frac{\partial \mu}{\partial \n}=0,\quad \frac{\partial c}{\partial \n}=0,\quad {\rm on}\,\,\partial\Omega.
\end{aligned}\right.
\end{eqnarray}
We  adopt  the argument in \cite{LW} and prove   \eqref{eq2.19} in the following  several steps.  

{\it Step 1.} It follows  directly  from $\eqref{eq2.20}_1$ that $\Vert\rho\Vert_{L^1}=m_1$. 

Multiplying $\eqref{eq2.20}_1$ by $\frac{\sigma}{2}|\u|^2$ and $\eqref{eq2.20}_2$ by $\u$ respectively,
 and use the Holder inequality,  we  arrive at 
\begin{eqnarray}\label{eq2.22}
\begin{aligned}
&\frac{\epsilon^2\sigma}{2}\int(\rho+\rho_0)|\u|^2+\sigma\epsilon^2\int\left(\frac{\left(\ln\frac{1}{\delta}\right)^{-1}\rho^{11}}{10}+ \rho^\gamma\right)\\
&\quad+\sigma\epsilon^4\int\left(\frac{4\left(\ln\frac{1}{\delta}\right)^{-1}}{11}|\na\rho^{\frac{11}{2}}|^2+\frac{4(\gamma-1)}{\gamma}|\na\rho^{\frac{\gamma}{2}}|^2\right)\\
&\quad+\int\mathbb{S}_{ns}:\na\u+\sigma\int\rho\frac{\partial f^\delta}{\partial c}(\u\cdot\na)c-\sigma\int\rho\mu(\u\cdot\na)c\\
&\leq \sigma\epsilon^2\int\left(\frac{\left(\ln\frac{1}{\delta}\right)^{-1}\rho_0^{11}}{10}+ \rho_0^\gamma\right)+\sigma\int(\rho \mathbf{g}_1+\mathbf{g}_2)\cdot\u+\sigma\int\rho H{\rm div}\u.
\end{aligned}
\end{eqnarray}
Multiplying $\eqref{eq2.20}_3$ by $\mu$ and $\eqref{eq2.20}_4$ by $\epsilon c$ yields
\begin{eqnarray}\label{eq2.23}
\int|\na\mu|^2+\epsilon\int|\na c|^2+\sigma \int\rho\mu(\u\cdot\na)c-\epsilon\sigma\int\rho_0c_0\mu+\epsilon\sigma\int\rho c\frac{\partial f^\delta}{\partial c}=0.
\end{eqnarray}
The combination of  \eqref{eq2.22} with \eqref{eq2.23} gives 
\begin{eqnarray}\label{eq2.242}
\begin{aligned}
&\frac{\epsilon^2\sigma}{2}\int(\rho+\rho_0)|\u|^2+\sigma\epsilon^2\int\left(\frac{\left(\ln\frac{1}{\delta}\right)^{-1}\rho^{11}}{10}+ \rho^\gamma\right)\\
& \quad +\sigma\epsilon^4\int\left(\frac{4\left(\ln\frac{1}{\delta}\right)^{-1}}{11}|\na\rho^{\frac{11}{2}}|^2+\frac{4(\gamma-1)}{\gamma}|\na\rho^{\frac{\gamma}{2}}|^2\right)\\
& \quad +\int|\na\mu|^2+\epsilon\int|\na c|^2+\int\mathbb{S}_{ns}:\na\u+\sigma\int\left(\eps \rho c  +  \rho\u\cdot\na c \right)\frac{\partial f^\delta}{\partial c}\\
&\leq\frac{\sigma\epsilon^2}{3}\int\left(\frac{\left(\ln\frac{1}{\delta}\right)^{-1}\rho_0^{11}}{10}+\rho_0^\gamma\right)+\sigma\int(\rho \mathbf{g}_1+\mathbf{g}_2)\cdot\u+\sigma\int\rho H{\rm div}\u+\epsilon\sigma\int\rho_0c_0\mu.
\end{aligned}
\end{eqnarray}
We now deal with the integrals in \eqref{eq2.242} individually.

By \eqref{eq1.4a}, it has \be\la{ls0}\frac{\partial f^\delta}{\partial c} = f_{2}^{\delta '}(c)-\theta_cc.\ee
 Hence, 
\begin{equation}\label{eq2.25}
\begin{aligned}
&\sigma\int\left(\eps \rho c  +  \rho\u\cdot\na c \right)\frac{\partial f^\delta}{\partial c}\\
&= \sigma\eps\int  \rho c \left(f_{2}^{\delta '}(c)-\theta_cc\right)+\sigma \int \rho\u\cdot\na c \left(f_{2}^{\delta '}(c)-\theta_cc\right).
\end{aligned}
\end{equation}
Using   $\eqref{eq2.20}_1$ once more,  
\begin{equation}\la{ls0}
\begin{aligned}
0&=-\sigma H\int\left({\rm div}(\rho\u) \ln \rho+ \u\cdot\na\rho\right)\\
&= \epsilon^2\sigma H\int(\rho-\rho_0) \ln\rho +4\epsilon^4 \sigma H\int |\na\sqrt{\rho}|^2   -\sigma H\int  \u\cdot\na\rho.
\end{aligned}
\end{equation}
By the fact $\Vert\rho\Vert_{L^1}=m_1$  and 
\be\la{ls1}|f_{2}^{\delta}(c)-\frac{\theta_c}{2} c^{2}|\le C(\de)(1+|c|),\quad |f_{2}^{\delta '}(c)-\theta_cc|\le C(\de),\ee  
we deduce  that  \begin{eqnarray*}
&&\left|\int\rho\u\cdot\na c\left(f_{2}^{\delta '}(c)-\theta_cc\right)\right|=\left|\int{\rm div}(\rho\u)\left(f_{2}^{\delta}(c)-\frac{\theta_c}{2}c^2\right)\right|\\
&&=\left|\epsilon^2\int(\rho-\rho_0) \left(f_{2}^{\delta}(c)-\frac{\theta_c}{2}c^2\right)+\epsilon^4\int\na\rho\cdot\na c \left(f_{2}^{\delta '}(c)-\theta_cc\right)\right|\\
&&\leq C(\delta)\epsilon^2\left(\int(1+\rho)(1+|c|)+\eps^{2} \int|\na\rho\cdot\na c|\right).
\end{eqnarray*}
The final integral works similarly,
\begin{eqnarray*}
&& \left|\sigma\eps\int  \rho c \left(f_{2}^{\delta '}(c)-\theta_cc\right)
\right|\leq C(\delta) \epsilon  \int\rho|c|.\end{eqnarray*}
 With above estimates in hand,  we substitute   \eqref{eq2.25} and \eqref{ls0}  into  \eqref{eq2.242} to obatin
  \begin{eqnarray}\label{eq2.242a}
\begin{aligned}
&\frac{\epsilon^2\sigma}{2}\int(\rho+\rho_0)|\u|^2+\sigma\epsilon^2\int\left(\frac{\left(\ln\frac{1}{\delta}\right)^{-1}\rho^{11}}{10}+ \rho^\gamma\right)\\
& 
\quad+\sigma\epsilon^4\int\left(4 H |\na\sqrt{\rho}|^2 +\frac{4\left(\ln\frac{1}{\delta}\right)^{-1}}{11}|\na\rho^{\frac{11}{2}}|^2+\frac{4(\gamma-1)}{\gamma}|\na\rho^{\frac{\gamma}{2}}|^2\right)\\
& \quad+\int|\na\mu|^2+\epsilon\int|\na c|^2+\lambda_{1}\int |\na\u|^{2}\\
&\leq  \epsilon^2\int\left(\frac{\left(\ln\frac{1}{\delta}\right)^{-1}\rho_0^{11}}{10}+ \rho_0^\gamma\right)+\int\left(\epsilon\rho_0c_0\mu +(\rho \mathbf{g}_1+\mathbf{g}_2)\cdot\u \right)+ \epsilon^2\sigma\int(\rho_0-\rho)H\ln\rho\\
&\quad+  \int H{\rm div}(\rho \u)+C\sigma\eps^{4} \int|\na\rho\cdot\na c|+C\sigma\epsilon\int(1+\rho)(1+|c|).
\end{aligned}
\end{eqnarray}

{\it Step 2.} 
Let us control the integrals  on the right-hand side of  \eqref{eq2.242a}. 

By \eqref{eq1.11} and \eqref{eq2.6}, we have 
\be\la{ls} \ba & \epsilon^2\int\left(\frac{\left(\ln\frac{1}{\delta}\right)^{-1}\rho_0^{11}}{10}+ \rho_0^\gamma\right)+\int\left(\epsilon\rho_0c_0\mu +(\rho \mathbf{g}_1+\mathbf{g}_2)\cdot\u \right)\\
&\le C\left(1+\|\mu\|_{L^{1}}+\|\rho\|_{L^{\frac{6}{5}}}^{2}\right)+\frac{\lambda_{1}}{2}\|\na \u\|_{L^{2}}^{2}.\ea\ee
The    monotonicity of  $\ln x$,   $\eqref{eq2.20}_1$ and    \eqref{eq1.12} guarantee that   
\be\la{ls4}
\begin{aligned}
&\epsilon^2\sigma \int H(\rho_0-\rho) \ln\rho  \\
&= \epsilon^2\sigma \int H(\rho_0-\rho)(\ln\rho - \ln\rho_0) - \epsilon^2\sigma \int H\rho \ln\rho_0   + \epsilon^2\sigma \int \rho_0 \ln\rho_0 H \\
& \leq C,
\end{aligned}
\ee
and 
\bnn\ba
\left|\int H{\rm div}(\rho\u)\right| =\sigma\left|-\epsilon^2\int H\rho +\epsilon^2\int\rho_0 H-\epsilon^4\int\Delta\rho H\right|\leq C. 
\ea\enn   
Hence, 
\be\la{ls5}
\begin{aligned}
&\left|\int H{\rm div}(\rho\u)+\sigma\eps^{4} \int|\na\rho\cdot\na c|\right|\\
&\leq C+C\sigma\epsilon^{6}\Vert\na\rho\Vert_{L^2}^2+\frac{\sigma\epsilon^2}{8}\Vert\na c\Vert_{L^2}^2\\
&\le  C+ \frac{\sigma\epsilon^4}{2}\int\left(2H|\na\sqrt{\rho}|^2 +\frac{2\left(\ln\frac{1}{\delta}\right)}{11}|\na\rho^{\frac{11}{2}}|^2 \right)+
 \frac{\sigma\epsilon^2}{8}\int|\na c|^2,
\end{aligned}\ee  where the last inequality sign is valid for small  $\eps=\eps(H,\,\de)>0$.

Next, thanks to   \eqref{ls0},  \eqref{ls1} and the  condition  $
\int\rho\mu=\int\rho\frac{\partial f^\delta}{\partial c}, 
$
\bnn\ba
\left|\int\mu\right|&=\left|\frac{1}{\rho_0}\int\rho\mu-\frac{1}{\rho_0}\int\rho(\mu-(\mu)_{\om})\right|\\
&\leq\frac{1}{\rho_0}\left|\int\rho\frac{\partial f^\delta}{\partial c}\right|+C\Vert\rho\Vert_{L^{\frac{6}{5}}}\Vert\na\mu\Vert_{L^2}\\
&\leq C\int\rho|f_{2}^{\delta '}(c)-\theta_cc|+C\Vert\rho\Vert_{L^{\frac{6}{5}}}\Vert\na\mu\Vert_{L^2}\\
&\leq C(\delta)\left(1+\Vert\rho\Vert_{L^{\frac{6}{5}}}\Vert\na\mu\Vert_{L^2}\right),
\ea\enn
and whence, for any $p\in[1,6]$,
\be\la{ls2}
\Vert\mu\Vert_{L^p}\leq C+C\left(1+\Vert\rho\Vert_{L^{\frac{6}{5}}}\right)\Vert\na\mu\Vert_{L^2}.
\ee
A similar method  gives
\bnn\ba
\left|\int c\right|&=\left|\frac{1}{\rho_0}\int\rho c-\frac{1}{\rho_0}\int\rho\left(c-(c)_\Omega\right)\right|\\
&\leq\left|\frac{m_2}{\rho_{0}}+\epsilon\int(\rho_0-\rho)c+\epsilon^3\int\rho\Delta c\right|+C\Vert\rho\Vert_{L^{\frac{6}{5}}}\Vert\na c\Vert_{L^2}\\
&\leq C+C\Vert\rho\Vert_{L^{\frac{6}{5}}}\Vert\na c\Vert_{L^2}+C\epsilon^3(\Vert\rho\Vert_{L^2}^2+\Vert\rho\Vert_{L^{\frac{12}{5}}}^2\Vert\mu\Vert_{L^6}).
\ea\enn
Thus, for  $p\in[1,6]$,
\be\la{ls6}
\Vert c\Vert_{L^p}\leq C+C(1+\Vert\rho\Vert_{L^{\frac{6}{5}}})\Vert\na c\Vert_{L^2}+C\epsilon^3(\Vert\rho\Vert_{L^2}^2+\Vert\rho\Vert_{L^{\frac{12}{5}}}^2\Vert\mu\Vert_{L^6}).
\ee
In terms of \eqref{ls2} and \eqref{ls6},  we compute 
\be\la{ls7}\ba 
C\epsilon\int(1+\rho)(1+|c|)
&\le C\epsilon\left(1+(1+\|\rho\|_{L^{\frac{6}{5}}})\|c\|_{L^{6}}\right)\\
&\le C\epsilon\left(1+ \|\rho\|_{L^{\frac{12}{5}}}^{6}\right)+ \frac{\epsilon}{8}\int|\na c|^2 +\frac{1}{2}\int|\na \mu|^2.\ea\ee
Therefore,  having   \eqref{ls},  \eqref{ls4},  \eqref{ls5},  and \eqref{ls7}  in hand,  we deduce from  \eqref{eq2.242a}  that  
  \begin{eqnarray}\label{eq2.224}
\begin{aligned}
&\frac{\epsilon^2\sigma}{2}\int(\rho+\rho_0)|\u|^2+\sigma\epsilon^2\int\left(\frac{\left(\ln\frac{1}{\delta}\right)^{-1}\rho^{11}}{10}+ \rho^\gamma\right)\\
& \quad
+\sigma\epsilon^4\int\left( H|\na\sqrt{\rho}|^2 +\frac{\left(\ln\frac{1}{\delta}\right)^{-1}}{11}|\na\rho^{\frac{11}{2}}|^2+\frac{4(\gamma-1)}{\gamma}|\na\rho^{\frac{\gamma}{2}}|^2\right)\\
&\quad +\frac{1}{2}\int|\na\mu|^2+\frac{\epsilon}{4}\int|\na c|^2+\frac{\lambda_{1}}{2}\int |\na\u|^{2}\\
&\leq C\left(1+  \|\rho\|_{L^{\frac{12}{5}}}^{6}\right).
\end{aligned}
\end{eqnarray}
As a consequence  of   \eqref{ls2},  \eqref{ls6},  \eqref{eq2.224}, \begin{equation}\label{eq2.246}
\Vert\mu\Vert_{L^p}+\Vert c\Vert_{L^p}\leq C+C \|\rho\|_{L^{\frac{12}{5}}}^{6}.\end{equation}
Here,  the constant  $C$ appeared in \eqref{eq2.224} and \eqref{eq2.246} is independent of $\eps$ and $\sigma.$

{\it Step 3.} By the interpolation theorem,   from \eqref{eq2.224} and \eqref{eq2.246}  we obtain 
\be\la{ls13}
\Vert\rho\Vert_{L^{11}}^{11}+\Vert\u\Vert_{H_0^1}^2+\Vert\mu\Vert_{H^1}^2+\Vert c\Vert_{H^1}^2+\Vert|\na\rho^2|+|\na\sqrt{\rho}|\Vert_{L^2}^2\leq C(\epsilon).
\ee
With  \eqref{ls13}, the  bootstrap argument ({\it cf}.\cite{LW}) shows
\be\la{ls9}\Vert\u\Vert_{W^{2,p}}+\Vert\mu\Vert_{W^{2,p}}+\Vert c\Vert_{W^{2,p}}\leq C(\epsilon),\quad p\in (1,\infty).\ee
Therefore, the desired \eqref{eq2.19} follows from \eqref{ls9}, and the proof of Theorem \ref{thm2.1} is thus complete. \end{proof}

\subsection{$\epsilon$-limit  for  the solutions}
We shall   take  $\epsilon\rightarrow0$ for the  solution $(\rho^{\eps},\u^{\eps},\mu^{\eps},c^{\eps})$  of  \eqref{eq2.7}-\eqref{eq2.8} and  obtain  the following result:
 \begin{theorem}\label{thm3.1} Let  assumptions  \eqref{eq1.11}-\eqref{eq1.12} hold true.  Then, for fixed $\de\in (0,1)$, the system  
\begin{equation}\label{eq3.1}
\left\{\begin{aligned}
&{\rm div}(\rho\u)=0,\\
&{\rm div}(\rho\u\otimes\u)+\na\left(\left(\ln\frac{1}{\delta}\right)^{-1}\rho^{11}+\rho^2\frac{\partial f^\delta}{\partial\rho}\right)={\rm div}(\mathbb{S}_{ns}+\mathbb{S}_c)+\rho\mathbf{g}_1+\mathbf{g}_2,\\
&{\rm div}(\rho\u c)=\Delta\mu,\\
&\rho\mu=\rho\frac{\partial f^\delta}{\partial c}-\Delta c,\\
&\int\rho=m_1,\quad\int\rho c=m_2,
\end{aligned}\right.\end{equation}
with the boundary conditions \eqref{eq1.8}  admits a weak solution  $(\rho^{\de},\u^{\de},\mu^{\de},c^{\de})$. Moreover, 
\begin{eqnarray}\label{eq3.1.2}
0\leq \rho^{\de}\in L^{12},\quad \u^{\de}\in H_0^1(\Omega;\mathbb{R}^3),\quad (\mu^{\de},c^{\de})\in H_{n}^1(\Omega)\times H_{n}^1(\Omega),
\end{eqnarray}
and the energy inequality \eqref{ls40} holds.
\end{theorem} 

Before proving Theorem \ref{thm3.1}, 
we   derive  some uniform in $\epsilon $ estimates on $(\rho^\epsilon,\u^\epsilon,\mu^\epsilon, c^\epsilon)$.
\begin{lemma}\label{lem3.1} The  solution $(\rho^{\eps},\u^{\eps},\mu^{\eps},c^{\eps})$ of  the  problem \eqref{eq2.7}-\eqref{eq2.8} satisfies  \begin{equation}\label{eq3.1a}
\Vert (\rho^\epsilon)^{12}\Vert_{L^1}+\Vert(\rho^\epsilon)^3\frac{\partial f^\delta(\rho^\epsilon,c^\epsilon)}{\partial\rho^\epsilon}\Vert_{L^1}+\Vert c^\epsilon\Vert_{H^1}\leq C,
\end{equation}
where the $C$ is independent of $\eps.$
\end{lemma}
\begin{proof} First we recall   the Bogovskii operator ({\it cf.}\cite{NS})
\begin{eqnarray}\label{eq2.36}
\mathcal{B}=[\mathcal{B}_1,\mathcal{B}_2,\mathcal{B}_3]:\,\,\left\{ f\in L^p|\int_{\om} f =0\right\}\to W_0^{1,p}(\Omega),\quad p\in(1,\infty).
\end{eqnarray}  Then,   ${\rm div}\mathcal{B}f=f\,\,a.e.\,\,\Omega$, and the following inequalities  hold true \begin{eqnarray}\label{eq2.37}
\Vert\na\mathcal{B}(f)\Vert_{L^p}\leq C\Vert f\Vert_{L^p},\quad \Vert\mathcal{B}(f)\Vert_{L^p}\leq C\Vert g\Vert_{L^p},
\end{eqnarray}  for $f={\rm div}g$ with  $g\in L^p$  and  $g\cdot \n|_{\partial\Omega}=0$. 

Multiplying  the momentum equation $\eqref{eq2.7}_{2}$ by $\mathcal{B}(\rho^\epsilon-\rho_0)$, we get 
\bnn\begin{aligned}
&\int\left(\left(\ln\frac{1}{\delta}\right)^{-1}(\rho^\epsilon)^{11}+(\rho^\epsilon)^2\frac{\partial f^\delta}{\partial \rho^\epsilon}\right)\rho^\epsilon\\
&=\int\left(\left(\ln\frac{1}{\delta}\right)^{-1}(\rho^\epsilon)^{11}+(\rho^\epsilon)^2\frac{\partial f^\delta}{\partial \rho^\epsilon}\right)\rho_0-\int(\rho^\epsilon \mathbf{g}_1+\mathbf{g}_2)\cdot\mathcal{B}(\rho^\epsilon-\rho_0)\\
&\quad+\epsilon^2\int\rho^\epsilon\u^\epsilon\cdot\mathcal{B}(\rho^\epsilon-\rho_0)+\epsilon^4\int\na\rho^\epsilon\cdot\na\u^\epsilon\mathcal{B}(\rho_\epsilon-\rho_0)-\int\rho^\epsilon\u^\epsilon\otimes\u^\epsilon:\na\mathcal{B}(\rho^\epsilon-\rho_0)\\
&\quad+\int\lambda_1(\na\u^\epsilon+(\na\u^\epsilon)^T):\na\mathcal{B}(\rho_\epsilon-\rho_0)+\lambda_2{\rm div}\u^\epsilon{\rm div}\mathcal{B}(\rho^\epsilon-\rho_0)\\
&\quad+\int\left(\rho^\epsilon\frac{\partial f^\delta}{\partial c^\epsilon}-\rho^\epsilon\mu^\epsilon\right)\mathcal{B}(\rho^\epsilon-\rho_0)\cdot \na c^\epsilon\\
&=:\sum_{i=1}^7I_i.
\end{aligned}\enn
We estimate them one by one. 
It follows from \eqref{eq1.12} that   \bnn\ba
I_1+I_2&\leq \left|\int\left(\left(\ln\frac{1}{\delta}\right)^{-1}(\rho^\epsilon)^{11}+(\rho^\epsilon)^2\frac{\partial f^\delta}{\partial\rho^\epsilon}\right)\rho_0\right|+\left|\int(\rho^\epsilon \mathbf{g}_1+\mathbf{g}_2)\cdot \mathcal{B}(\rho^\epsilon-\rho_0)\right|\\
&\leq\frac{1}{8}\int\left(\left(\ln\frac{1}{\delta}\right)^{-1}(\rho^\epsilon)^{12}+(\rho^\epsilon)^3\frac{\partial f^\delta}{\partial\rho^\epsilon}\right)+C,
\ea\enn
and from  \eqref{eq2.224} that  
\bnn\ba
I_3+I_4+I_5&\leq \left(\epsilon^2\Vert\rho^\epsilon\Vert_{L^2}+\epsilon^4\Vert\na\rho^\epsilon\Vert_{L^2}\right)\Vert\u^\epsilon\Vert_{H^1}\Vert\mathcal{B}(\rho^\epsilon-\rho_0)\Vert_{L^\infty}\\
&\quad+\Vert\rho^\epsilon\Vert_{L^{\frac{12}{5}}}\Vert\u^\epsilon\Vert_{L^6}^2\Vert\na\mathcal{B}(\rho^\epsilon-\rho_0)\Vert_{L^4}\\
&\leq \frac{1}{8}\int\left(\ln\frac{1}{\delta}\right)^{-1}(\rho^\epsilon)^{12} +C,
\ea\enn
\bnn\ba
I_6&\leq\left|\int\lambda_1(\na\u^\epsilon+(\na\u^\epsilon)^T):\na\mathcal{B}(\rho^\epsilon-\rho_0)+\lambda_2{\rm div}\u^\epsilon{\rm div}\mathcal{B}(\rho^\epsilon-\rho_0)\right|\\
&\leq  \frac{1}{8}\int \left(\ln\frac{1}{\delta}\right)^{-1}(\rho^\epsilon)^{12} +C.
\ea\enn
To deal with the last term, we multiply $\eqref{eq2.16}_3$ by $c^\epsilon$ and deduce 
\bnn\ba
\Vert\na c^\epsilon\Vert_{L^2}^2&=\int\left(\rho^\epsilon\mu^\epsilon-\rho^\epsilon\frac{\partial f^\delta}{\partial c^\epsilon}\right)c^\epsilon\\ 
&\leq C(\delta)\Vert c^\epsilon\Vert_{L^2}\left(\Vert\rho^\epsilon\Vert_{L^3}\Vert\mu^\epsilon\Vert_{L^6}+\Vert\rho^\epsilon\Vert_{L^2}\right)\\
&\le\frac{1}{2}\Vert\na c^\epsilon\Vert_{L^2}^2+C(\delta)(\Vert\rho^\epsilon\Vert_{L^5}^{4}+1),
\ea\enn
where the last inequality is due to   \eqref{eq2.246}.
Thus,
\bnn\ba
I_7&\leq  C\Vert\mathcal{B}(\rho^\epsilon-\rho_0)\Vert_{L^\infty}\Vert\na c^\epsilon\Vert_{L^2}\left(\int\left|\rho^\epsilon\mu^\epsilon-\rho^\epsilon\frac{\partial f^\delta}{\partial c^\epsilon}\right|^2\right)^{\frac{1}{2}}\\
&\leq  C(\delta)\Vert \mathcal{B}(\rho^\epsilon-\rho_0)\Vert_{L^\infty}\Vert\na c^\epsilon\Vert_{L^2}(\Vert\rho^\epsilon\Vert_{L^3}\Vert\mu^\epsilon\Vert_{L^6}+\Vert\rho^\epsilon\Vert_{L^2})\\
&\leq \frac{1}{8}\int \left(\ln\frac{1}{\delta}\right)^{-1}(\rho^\epsilon)^{12} +C.
\ea\enn
Taking the above inequalities into accounts,  we conclude  
\[\Vert\na c^\epsilon\Vert_{L^2}+\int\left(\left(\ln\frac{1}{\delta}\right)^{-1}(\rho^\epsilon)^{12}+(\rho^\epsilon)^3\frac{\partial f^\delta}{\partial \rho^\epsilon}\right)\leq C.\]
The proof of Lemma \ref{lem3.1} is thus completed.\end{proof}

\bigskip

Now we are in a position to prove Theorem \ref{thm3.1}.  In fact, 
the \eqref{eq3.1.2}  follows directly from   Lemma \ref{lem3.1} and  \eqref{eq2.224}-\eqref{eq2.246}.  Furthermore, \be\ba\label{eq3.3}
&\rho^\epsilon\,-\!\!\!\rightharpoonup\,\rho\quad{\rm in}\,\, L^{12}\cap L^{\gamma+1},\quad (\rho^\epsilon)^{11} \,-\!\!\!\rightharpoonup\,\overline{\rho^{11}}\quad {\rm in}\,\,L^{\frac{12}{11}},
 \\
 &(\na\u^\epsilon,\na\mu^\epsilon,\na c^\epsilon)\,-\!\!\!\rightharpoonup\,(\na\u,\na\mu,\na c)\quad {\rm in}\,\,\, L^2,
\\&
(\u^\epsilon,\mu^\epsilon,c^\epsilon)\longrightarrow(\u,\mu,c)\quad {\rm in}\,\,\, L^{p_1}\quad {\rm for}\,\, p_1\in [1,6),
\\&
\epsilon^4\na\rho^\epsilon\longrightarrow 0\quad {\rm in}\,\,L^2,
\\&
\epsilon^2\rho^\epsilon\longrightarrow 0,\quad\epsilon^2\rho^\epsilon\u^\epsilon\longrightarrow 0,\quad\epsilon\rho^\epsilon c^\epsilon\longrightarrow 0,\quad\epsilon^4\na\rho^\epsilon\na\u^\epsilon\longrightarrow 0\quad {\rm in}\,\,L^1.
\ea\ee
Consequently, for some $p>1$,
\begin{eqnarray}\label{eq3.8}
(\rho^\epsilon\u^\epsilon,\rho^\epsilon\mu^\epsilon)\,-\!\!\!\rightharpoonup\,(\rho\u,\rho\mu)\quad {\rm in}\,\, L^2,\quad \rho^\epsilon\u^\epsilon\otimes\u^\epsilon\,-\!\!\!\rightharpoonup\,\rho\u\otimes\u\quad {\rm in}\,\,L^p.
\end{eqnarray}
Due to the condition \eqref{eq1.12}, the definition of $f^{\de}$,  and \eqref{eq3.3},   we have 
\be\ba\label{eq3.9}
&\quad \rho^\epsilon\frac{\partial f^\delta}{\partial c^\epsilon}= \rho^\epsilon (f_{2}^{\delta '}(c^\epsilon)-\theta_cc^\epsilon) \,-\!\!\!\rightharpoonup \, \rho (f_{2}^{\delta '}(c)-\theta_c c)=\rho\frac{\partial f^\delta}{\partial c}\,\,{\rm in}\,\, L^{2},
\\&
(\rho^\epsilon)^2\frac{\partial f^\delta}{\partial \rho^\epsilon}=(\gamma-1)(\rho^\epsilon)^\gamma+\rho^\epsilon H\,-\!\!\!\rightharpoonup\,(\gamma-1)\overline{\rho^\gamma}+\rho H=\overline{\rho^2\frac{\partial f^\delta}{\partial \rho}}\,\,{\rm in}\,\, L^{\frac{\gamma+1}{\gamma}}.
\ea\ee

Next,  by \eqref{eq3.3}, \eqref{eq3.9},   if  we  multiply  $\eqref{eq2.7}_4$ by $c^{\eps}$  and by $c$ respectively, we obtain
\[\lim_{\epsilon\to 0}\int|\na c^\epsilon|^2=\lim_{\epsilon\to 0}\int\rho^\epsilon\mu^\epsilon c^\epsilon-\lim_{\epsilon\to 0}\int\rho^\epsilon\frac{\partial f^\delta}{\partial c^\epsilon}c^\epsilon=\int\rho\mu c-\int\rho\frac{\partial f^\delta}{\partial c}c,\]
and \[\lim_{\epsilon\to 0}\int\na c^\epsilon\cdot\na c=\lim_{\epsilon\to 0}\int\rho^\epsilon\mu^\epsilon c-\lim_{\epsilon\to 0}\int\rho^\epsilon\frac{\partial f^\delta}{\partial c^\epsilon}c=\int\rho\mu c-\int\rho\frac{\partial f^\delta}{\partial c}c.\]
Thus, 
\[ \lim_{\epsilon\to 0}\int|\na c^\epsilon|^2  =\lim_{\epsilon\to 0}\int\na c^\epsilon\cdot\na c= \int|\na c|^2,\]
which together with  \eqref{eq3.3} guarantee that   $\na c^\epsilon$ strongly converges to  $\na c.$
\medskip

We can take the $\epsilon$ limit for $(\rho^{\eps},\u^{\eps},\mu^{\eps},c^{\eps})$ to complete the proof of Theorem \ref{thm3.1}, on the condition that
\be\la{ls11}
\rho^\epsilon\longrightarrow \rho\quad {\rm in}\,\, L^1.
\ee
Let us prove  \eqref{ls11}.   In fact,   the compactness results in \eqref{eq3.3} guarantee that  the limit $(\rho,\u)$ satisfies 
\be\la{ls15}
{\rm div}(\rho\u)=0\quad {\rm in}\,\, \mathcal{D}'(\om).
\ee
Extend $(\rho,\u)$ by zero to the whole space; the same argument as in \cite{FNP} shows
\begin{eqnarray}\label{ls12}
{\rm div}(b(\rho)\u)+(b'(\rho)\rho-b(\rho)){\rm div}\u=0\quad {\rm in}\,\, \mathcal{D}'(\mathbb{R}^3),
\end{eqnarray}
where $b(z)=z$, or $b\in C^1[0,\infty)$ with $b'(z)=0$ for large $z$.
Define
\begin{eqnarray*}
C^2[0,\infty)\ni b_n(\rho)=\left\{\begin{aligned}
&\rho\ln\left(\rho+\frac{1}{n}\right),\quad\rho\leq n,\\
&(n+1)\ln\left(n+1+\frac{1}{n}\right),\quad\rho\geq n+1.
\end{aligned}\right.
\end{eqnarray*}
 Choosing $b_n$ in \eqref{ls12} and passing  $n\to \infty$, we  have 
${\rm div}(\u\rho\ln\rho)+\rho{\rm div}\u=0,$
and thus,
\begin{eqnarray}\label{eq4.271}
\int\rho{\rm div}\u=0.
\end{eqnarray}
If we multiply $\eqref{eq2.7}_1$ by $b_n'(\rho^\epsilon)$, we infer 
\begin{equation}\label{eq3.133}
\begin{aligned}
&\int(b_n'(\rho^\epsilon)\rho^\epsilon-b_n(\rho^\epsilon)){\rm div}\u^\epsilon\\
&=\int b_n'(\rho^\epsilon) {\rm div}(\rho^{\eps}\u^\epsilon)= \epsilon^2\int\rho_0 b_n'(\rho^\epsilon)-\epsilon^2\int\rho^\epsilon b_n'(\rho^\epsilon)-\epsilon^4\int b_n''(\rho^\epsilon)|\na\rho^\epsilon|^2\\
&\leq \epsilon^2\int\rho_0b_n'(\rho^\epsilon)-\epsilon^2\int\rho^\epsilon b_n'(\rho^\epsilon)-\epsilon^4\int_{\{x:b_n''\leq 0\}}b_n''(\rho^\epsilon)|\na\rho^\epsilon|^2.
\end{aligned}
\end{equation}
By  \eqref{ls13}, it has    $\Vert\na\rho^\epsilon\Vert_{L^2}\leq  C(\epsilon)$.  Thus,   the definition of $b_n$ implies 
\begin{equation}\label{eq3.14}
\begin{aligned}
-\epsilon^4\int_{\{x:b_n''\leq 0\}}b_n''(\rho^\epsilon)|\na\rho^\epsilon|^2\leq C(\epsilon)\int_{\{x:n\leq \rho^\epsilon\leq n+1\}}|\na\rho^\epsilon|^2\to 0\quad {\rm as}\,\,\,n\to \infty.
\end{aligned}
\end{equation} 
Thanks to  \eqref{eq3.1a}, there exists some constant $C$  uniformly in $\eps$ such that
\be\la{ls14}\left|\eps^{2}\int\rho_0b_n'(\rho^\epsilon)  +\eps^{2}\int\rho^\epsilon b_n'(\rho^\epsilon)\right|\leq C\eps^{2}\quad {\rm as}\,\,\,n\to \infty.\ee
Thus,  combining equalities  \eqref{eq3.133}-\eqref{ls14}  with   \eqref{eq4.271}, we arrive at 
\begin{eqnarray}\label{eq4.301}
\lim_{\epsilon\to 0}\int\rho^\epsilon{\rm div}\u^\epsilon =\lim_{\epsilon\to 0}\lim_{n\to \infty}\int(b_n'(\rho^\epsilon)\rho^\epsilon-b_n(\rho^\epsilon)){\rm div}\u^\epsilon\leq 0=\int\rho{\rm div}\u.
\end{eqnarray}

\begin{lemma}\label{lem3.3} The following holds: 
\begin{eqnarray}\label{eq3.16}
\lim_{\epsilon\to 0}\int\phi\rho^\epsilon \,\mathbb{F}(\rho^{\eps},\u^{\eps})=\int\phi\rho \,\overline{\mathbb{F}(\rho,\u)},\quad \forall\phi\in C_0^\infty(\Omega),
\end{eqnarray}
where  
\[\mathbb{F}(\rho,\u):= \left(\ln\frac{1}{\delta}\right)^{-1}\rho^{11}+\rho^2\frac{\partial f^\delta(\rho,c)}{\partial \rho}-(2\lambda_1+\lambda_2){\rm div}\u,\]
\[\overline{\mathbb{F}(\rho,\u)}= \left(\ln\frac{1}{\delta}\right)^{-1}\overline{\rho^{11}}+\overline{\rho^2\frac{\partial f^\delta(\rho,c)}{\partial \rho}}-(2\lambda_1+\lambda_2){\rm div}\u.\]
\end{lemma}
\begin{proof} The proof is provided in the Appendix.  \end{proof}

On account  of    \eqref{eq4.301} and  \eqref{eq3.16}, one has 
\begin{eqnarray}\label{eq3.24}
\lim_{\epsilon\to 0}\int\rho^\epsilon\left(\left(\ln\frac{1}{\delta}\right)^{-1}(\rho^\epsilon)^{11}+(\rho^\epsilon)^2\frac{\partial f^\delta}{\partial\rho^\epsilon}\right)\leq \int\rho\left(\left(\ln\frac{1}{\delta}\right)^{-1}\overline{\rho^{11}}+\overline{\rho^2\frac{\partial f^\delta}{\partial\rho}}\right).
\end{eqnarray}
Recalling  the definition of $f^{\de},$ we use   \eqref{eq3.24} and compute
\begin{eqnarray*}
&&\int\left(\left(\ln\frac{1}{\delta}\right)^{-1}\overline{\rho^{12}}+(\gamma-1)\overline{\rho^{\gamma+1}}+\overline{\rho^2}H\right)\\
&&=\lim_{\epsilon\to 0}\int\rho^\epsilon\left(\left(\ln\frac{1}{\delta}\right)^{-1}(\rho^\epsilon)^{11}+(\rho^\epsilon)^2\frac{\partial f^\delta}{\partial\rho^\epsilon}\right)\\
&&\leq\int\rho\left(\left(\ln\frac{1}{\delta}\right)^{-1}\overline{\rho^{11}}+\overline{\rho^2\frac{\partial f^\delta}{\partial\rho}}\right)=\int\rho\left(\left(\ln\frac{1}{\delta}\right)^{-1}\overline{\rho^{11}}+(\gamma-1)\overline{\rho^\gamma}+\rho H\right),
\end{eqnarray*}
which implies 
\begin{eqnarray}\label{eq3.25}
\int\left(\ln\frac{1}{\delta}\right)^{-1}(\rho\overline{\rho^{11}}-\overline{\rho^{12}})\geq (\gamma-1)\int(\overline{\rho^{\gamma+1}}-\rho\overline{\rho^\gamma})+\int(\overline{\rho^2}-\rho^2)H\geq 0.
\end{eqnarray}   By monotonicity,   for any  constant $\beta>0$,
\begin{eqnarray*}
0&\leq&\int(\rho_\epsilon^{11}-(\rho+\beta\eta)^{11})(\rho_\epsilon-(\rho+\beta\eta))\\
&=&\int\left((\rho^\epsilon)^{12}-(\rho^\epsilon)^{11}\rho-(\rho^\epsilon)^{11}\beta\eta-(\rho+\beta\eta)^{11}\rho^\epsilon+(\rho+\beta\eta)^{12}\right).
\end{eqnarray*}
This  together with  \eqref{eq3.25} imply that, by passing     $\epsilon\to 0$,
\[0\leq \int\left(\overline{\rho^{12}}-\rho\overline{\rho^{11}}-\overline{\rho^{11}}\beta\eta+(\rho+\beta\eta)^{11}\beta\eta\right)\leq \int\left(-\overline{\rho^{11}}+(\rho+\beta\eta)^{11}\right)\beta\eta.\]
If we replace $-\beta$ with $\beta$ in above  and  take the limit  $\beta\to 0$, we  infer 
\bnn \int\left(\rho^{11}-\overline{\rho^{11}}\right)\eta=0.\enn
Hence, $\overline{\rho^{11}}=\rho^{11}$ due to  the   arbitrariness of $\eta$.    This   proves 
  \eqref{ls11}.   

\bigskip

\section{Approximate System   \eqref{eq3.1}-\eqref{eq3.1.2}}

In this section, we consider the system \eqref{eq3.1}-\eqref{eq3.1.2} and derive the uniform-in-$\delta$ estimates. 

\subsection{Uniform in $\delta$ estimates}

Throughout this section, we will drop the superscript and denote $(\rho^\delta,\u^\delta,\mu^\delta,c^\delta)$ as $(\rho,\u,\mu,c)$ for simplicity.  

\begin{proposition}\la{pro4} In addition to  the assumptions  in Theorem \ref{t},  assume that 
\be\la{ls70}
 s\in \left(1,\,\,3-\frac{3}{\gamma}\right].
\ee
Then the solution $(\rho,\u,\mu,c)$     of    \eqref{eq3.1} satisfies 
\begin{equation}\label{eq24}
\begin{aligned}
\left\Vert \left(\ln\frac{1}{\delta}\right)^{-1}\rho^{11}+\rho^\gamma+\rho|\u|^2\right\Vert_{L^s}+\Vert\mu\Vert_{W^{1,2}}^2+\Vert c\Vert_{W^{2,\frac{6s}{5}}}^2+\Vert\u\Vert_{W_0^{1,2}}\leq C,
\end{aligned}
\end{equation}
where, and in the following, the constant $C$ is independent of $\de.$
\end{proposition}
\begin{proof} The proof is divided into several steps.

{\it Step 1.}    Multiplying  $\eqref{eq3.1}_{4}$  by $c$ yields
\begin{equation}\label{eqsh11}
\begin{aligned}
\int\rho(f_2^{\delta'}-\theta_cc)c+\int|\na c|^2=\int\rho\mu c.
\end{aligned}
\end{equation}
By \eqref{eq1.8},  we have  \begin{eqnarray}\label{eqccc}
m_{1}|(c)_\Omega|= \left| \int\rho c -\int\rho (c-(c)_\Omega)\right|\leq |m_{2}|+C\Vert\rho\Vert_{L^{\frac{6}{5}}}\Vert\na c\Vert_{L^2},
\end{eqnarray}
and hence, 
\begin{eqnarray}\label{eqccc1}
\Vert c\Vert_{L^p}\leq C+C\left(1+\Vert\rho\Vert_{L^{\frac{6}{5}}}\right)\Vert\na c\Vert_{L^2},\quad p\in [1,6].
\end{eqnarray}
Let 
\be\la{ls80}
m=\max\left\{c^*,\frac{\frac{|m_2|}{m_1}+1}{2}\right\},
\ee
where  $c^*$ is taken from \eqref{eq2.3}.  Then, using \eqref{dan}, we  integrate   $\eqref{eq3.1}_{4}$  and deduce 
 \begin{equation}\label{mu_mean3}
\begin{aligned}
\left|\int\rho\mu\right|&=\left|\int\rho\left(f_2^{\delta'}-\theta_cc\right)\right|\\
&\leq \left|\int_{\{x:|c|\leq m\}}\rho\left(f_2^{\delta'}-\theta_cc\right)\right|+\left|\int_{\{x:|c|> m\}}\rho\left(f_2^{\delta'}-\theta_cc\right)\right|\\
&\leq C+\frac{1}{m}\int_{\{x:|c|> m\}}\rho\left(f_2^{\delta'}-\theta_cc\right)c.
\end{aligned}
\end{equation}
 Hence,   
\begin{equation}\label{eq20}
\begin{aligned}
\left|m_2(\mu)_\Omega\right|&=\left|m_2\frac{\int\rho\mu}{\int\rho}-m_2\frac{\int\rho(\mu-(\mu)_\Omega)}{\int\rho}\right|\\
&\leq \frac{|m_2|}{m_1}\left(\left|\int\rho\mu\right|+C\Vert\na\mu\Vert_{L^2}\Vert\rho\Vert_{L^{\frac{6}{5}}}\right)\\
&\leq \frac{|m_2|}{m_1}\left(C+\frac{1}{m}\int_{\{x:|c|>m\}}\rho(f_2^{\delta'}-\theta_cc)c+C\Vert\rho\Vert_{L^{\frac{6}{5}}}^2\right),
\end{aligned}
\end{equation}
where, for the inequalities, we have used \eqref{mu_mean3},  
  and the  following inequality
\begin{eqnarray}\label{eq2229}
\Vert\na\u\Vert_{L^2}^2+\Vert\na\mu\Vert_{L^2}^2\leq C+ C\Vert\rho\Vert_{L^{\frac{6}{5}}}^2,
\end{eqnarray}
which comes from the energy inequality  \eqref{ls40} and the assumption \eqref{eq1.11g}.

Next,  utilizing    \eqref{eqccc}, \eqref{eq2229},    and the  interpolation 
$$\Vert\rho\Vert_{L^{\frac{6}{5}}}^2\le \Vert\rho\Vert_{L^{1}}\Vert\rho\Vert_{L^{\frac{3}{2}}}\le m_{1}\Vert\rho\Vert_{L^{\frac{3}{2}}},$$
we compute
\begin{equation}\label{eq229}
\begin{aligned}
&\left|\int\rho c(\mu-(\mu)_\Omega)\right|\\
&\leq C\Vert\na\mu\Vert_{L^2}\left(\Vert\rho\Vert_{L^{\frac{6}{5}}(\{x:|c|<1+\delta\})}
+\Vert\rho c\Vert_{L^{\frac{6}{5}}(\{x:|c|\geq 1+\delta\})}\right)\\
&\leq C+C\Vert\rho\Vert_{L^{\frac{6}{5}}}^2\\
&\quad +C\Vert\na\mu\Vert_{L^2}\left(\Vert\rho(c-(c)_\Omega)\Vert_{L^{\frac{6}{5}}(\{x:|c|\geq 1+\delta\})}+|(c)_\Omega|\Vert\rho\Vert_{L^{\frac{6}{5}}(\{x:|c|\geq 1+\delta\})}\right)\\
&\leq C+C\Vert\rho\Vert_{L^{\frac{6}{5}}}^2\\
&\quad+C\Vert\na\mu\Vert_{L^2}\left(\Vert \na c\Vert_{L^2}\Vert\rho\Vert_{L^{\frac{3}{2}}(\{x:|c|\geq 1+\delta\})} +\Vert\rho\Vert_{L^{\frac{6}{5}}}\Vert\na c\Vert_{L^2}\Vert\rho\Vert_{L^{\frac{6}{5}}(\{x:|c|\geq 1+\delta\})}\right)\\
&\leq C+\frac{1}{2}\Vert\na c\Vert_{L^2}^2+C\Vert\rho\Vert_{L^{\frac{6}{5}}}^2\left(1+\Vert\rho\Vert_{L^{\frac{3}{2}}(\{x:|c|\geq 1+\delta\})}^2+\Vert\rho\Vert_{L^{\frac{6}{5}}}^2\Vert\rho\Vert_{L^{\frac{3}{2}}(\{x:|c|\geq 1+\delta\})}\right).
\end{aligned}
\end{equation}
With  \eqref{eq20} and \eqref{eq229} in hand,   we estimate the last  integrals in  \eqref{eqsh11} as  
\begin{equation}\label{eqrhomuc1}
\begin{aligned}
\left|\int\rho\mu c\right|&=\left|\int\rho c(\mu-(\mu)_\Omega)+m_2(\mu)_\Omega\right|\\
&\leq C+\frac{1}{2}\Vert\na c\Vert_{L^2}^2+\frac{|m_2|}{mm_1}\int_{\{x:|c|>m\}}\rho(f_2^{\delta'}-\theta_cc)c\\
&\quad+C\Vert\rho\Vert_{L^{\frac{6}{5}}}^2\left(1+\Vert\rho\Vert_{L^{\frac{3}{2}}(\{x:|c|\geq 1+\delta\})}^2+\Vert\rho\Vert_{L^{\frac{6}{5}}}^2\Vert\rho\Vert_{L^{\frac{3}{2}}(\{x:|c|\geq 1+\delta\})}\right).
\end{aligned}
\end{equation}

Next, we shall deal with   $\Vert\rho\Vert_{L^{\frac{3}{2}}(\{x:|c|\geq 1+\delta\})}$ and bound the last two terms    in \eqref{eqrhomuc1}.   
Note that the definition of $f_2^{\delta}$ implies   that if  $c>1+\delta$, then
\[f_2^{\delta'}-\theta_cc=-(1+\delta)\theta+\frac{\theta_c\delta}{2}+\frac{3\theta_0}{2(2-\delta)}+\frac{\theta_0}{2}\ln\left(\frac{2-\delta}{\delta}\right)=O\left(\ln\frac{1}{\delta}\right)\quad {\rm as}\,\,\delta\to 0.\]
Thus, \[\Vert\rho\Vert_{L^1(\{x:|c|\geq 1+\delta\})}\leq (1+\de)^{-1}\left(\ln\frac{1}{\delta}\right)^{-1}\Vert\rho(f_{2}^{\delta'}-\theta_cc)c\Vert_{L^1(\{x:|c|\geq 1+\delta\})}.\]
Consequently,  by the interpolation inequality, 
\bnn\label{eq2299}
\Vert\rho\Vert_{L^{\frac{3}{2}}(\{x:|c|\geq 1+\delta\})}^2\leq C\left(\ln\frac{1}{\delta}\right)^{-\frac{1}{4}}\Vert\rho(f_2^{\delta'}-\theta_cc)c\Vert_{L^1(\{x:|c|\geq 1+\delta\})}^{\frac{1}{4}}\Vert\rho\Vert_{L^{\frac{21}{13}}}^{\frac{7}{4}}.
\enn
Then one further deduces that 
\begin{equation}\label{eq230}
\begin{aligned}
&C\Vert\rho\Vert_{L^{\frac{6}{5}}}^2\Vert\rho\Vert_{L^{\frac{3}{2}}(\{x:|c|\geq 1+\delta\})}^2\\
&\leq C\left(\ln\frac{1}{\delta}\right)^{-\frac{1}{4}}\Vert\rho\Vert_{L^{\frac{6}{5}}}^2\Vert\rho(f_2^{\delta'}-\theta_cc)\Vert_{L^1(\{x:|c|\geq 1+\delta\})}^{\frac{1}{4}}\Vert\rho\Vert_{L^{\frac{21}{13}}}^{\frac{7}{4}}\\
&\leq \frac{1}{8}\left(1-\frac{|m_2|}{mm_1}\right)\Vert\rho (f_2'-\theta_cc)c\Vert_{L^1}+C\left(\ln\frac{1}{\delta}\right)^{-\frac{1}{3}}\Vert\rho\Vert_{L^4}^{\frac{16}{9}}+C.
\end{aligned}
\end{equation}
Similarly,   
\begin{equation}\label{eq231}
\begin{aligned}
&C\Vert\rho\Vert_{L^{\frac{3}{2}}}^2\Vert\rho\Vert_{L^{\frac{3}{2}}(\{x:|c|\geq 1+\delta\})}\\
&\leq \frac{1}{8}\left(1-\frac{|m_2|}{mm_1}\right)\Vert\rho (f_2'-\theta_cc)c\Vert_{L^1}+C\left(\ln\frac{1}{\delta}\right)^{-1}\Vert\rho\Vert_{L^{3}}^3+C.
\end{aligned}
\end{equation}
The combination of   \eqref{eq230} with   \eqref{eq231} guarantees that   \eqref{eqrhomuc1} satisfies
\begin{equation}\label{eqrhomuc2}
\begin{aligned}
\left|\int\rho\mu c\right|\leq&C+\frac{|m_2|}{mm_1}\int_{\{x:|c|>m\}}\rho(f_2'-\theta_cc)c+\frac{1}{2}\Vert\na c\Vert_{L^2}^2+C\Vert\rho\Vert_{L^{\frac{6}{5}}}^{2}\\
&+\frac{1}{4}\left(1-\frac{|m_2|}{mm_1}\right)\Vert\rho (f_2'-\theta_cc)c\Vert_{L^1}+C\left(\ln\frac{1}{\delta}\right)^{-\frac{1}{3}}\Vert\rho\Vert_{L^4}^{\frac{8}{3}}.
\end{aligned}
\end{equation}
Notice that
$$\int_{\{x:|c|\le m\}}\rho\left|(f_2'-\theta_cc)c\right|\le C\int_{\{x:|c|\le m\}}\rho \le Cm_{1}$$
and 
$$\frac{|m_2|}{mm_1}+\frac{1}{4}\left(1-\frac{|m_2|}{mm_1}\right)= \frac{1}{4}\left(1+3\frac{|m_2|}{mm_1}\right)<1,$$
we then substitute \eqref{eqrhomuc2} into \eqref{eqsh11}  and conclude 
\begin{equation}\label{eqsh1}
\begin{aligned}
\int\rho|(f_2^{\delta'}-\theta_cc)c|+\int|\na c|^2\leq C+C\Vert\rho\Vert_{L^{\frac{6}{5}}}^2+C\left(\ln\frac{1}{\delta}\right)^{-\frac{1}{3}}\Vert\rho\Vert_{L^4}^{\frac{8}{3}}.
\end{aligned}
\end{equation}

{\it Step 2.} The same deduction of \eqref{eq20}  gives \begin{equation}
\begin{aligned}
|(\mu)_\Omega| \leq C+C\int\rho|(f_2^{\delta'}-\theta_cc)c|+C\Vert\rho\Vert_{L^{\frac{6}{5}}}^2,
\end{aligned}
\end{equation}
and hence,
\be\la{tty}
\Vert\mu\Vert_{L^6}\leq C+C\Vert\na\mu\Vert_{L^2}+C\int\rho|f_2^{\delta'}-\theta_cc)|c+C\Vert\rho\Vert_{L^{\frac{6}{5}}}^2.
\ee
Then,  it follows from  \eqref{eq2229},  \eqref{eqsh1}, \eqref{tty} that
\begin{eqnarray}\label{eqsh13}
\Vert\mu\Vert_{W^{1,2}}^2\leq C+C\Vert\rho\Vert_{L^{\frac{6}{5}}}^4+C\left(\ln\frac{1}{\delta}\right)^{-\frac{2}{3}}\Vert\rho\Vert_{L^4}^{\frac{16}{3}}.
\end{eqnarray}

Multiplying $\eqref{eq3.1}_{4}$ by $\frac{\p f^{\de}}{\p c}=f_2^{\delta'}-\theta_cc$ yields
\bnn\ba
\int\rho(f_2^{\delta'}-\theta_cc)^2+\int(f_2^{\delta''}-\theta_c)|\na c|^2&\leq \int\rho\mu^2\le 
 \Vert\rho\Vert_{L^{\frac{3}{2}}}\Vert\mu\Vert_{L^6}^2.
\ea\enn This together  with \eqref{eqsh13} provide us
\begin{equation}\label{eqsh15}
\begin{aligned}
&\int\rho(f_2^{\delta'}-\theta_cc)^2+\int\rho\mu^2 +\int|f_2^{\delta''}-\theta_c||\na c|^2\leq   C+C\Vert\rho\Vert_{L^{\frac{3}{2}}}^3+C\left(\ln\frac{1}{\delta}\right)^{-\frac{2}{3}}\Vert\rho\Vert_{L^4}^{\frac{52}{9}}.
\end{aligned}
\end{equation}

{\it Step 3.} We claim that   \begin{eqnarray}\label{rho_4}
\left\Vert\left(\ln\frac{1}{\delta}\right)^{-1}\rho^{11}+\rho^2\frac{\partial f^{\de}}{\partial\rho}\right\Vert_{L^s}\leq C.
\end{eqnarray}  
In fact,  let   the operator $\mathcal{B}$ be as in \eqref{eq2.36} and  $h\in L^{\frac{s}{s-1}}$ will be specified later. Then,  for  $1<s<\frac{3}{2}$,  it holds  that \begin{eqnarray*}
\Vert h\Vert_{L^2}+\Vert\mathcal{B}(h-(h)_\Omega)\Vert_{L^\infty}+\Vert\na\mathcal{B}(h-(h)_\Omega)\Vert_{L^2}+\Vert\na\mathcal{B}(h-(h)_\Omega)\Vert_{L^{\frac{s}{s-1}}}\leq C\Vert h\Vert_{L^{\frac{s}{s-1}}}.
\end{eqnarray*}
Multiplying    $\eqref{eq3.1}_2$  by  $\mathcal{B}(h-(h)_\Omega)$, we deduce
\bnn\begin{aligned}
&\int\left(\left(\ln\frac{1}{\delta}\right)^{-1}\rho^{11}+\rho^2\frac{\partial f^{\de}}{\partial \rho}\right)h\\
& = (h)_\Omega\int\left(\left(\ln\frac{1}{\delta}\right)^{-1}\rho^{11}+\rho^2\frac{\partial f^{\de}}{\partial \rho}\right)-\int(\rho \mathbf{g}_1+\mathbf{g}_2)\cdot\mathcal{B}(h-(h)_\Omega)+\int\mathbb{S}_{ns}:\na\mathcal{B}(h-(h)_\Omega)\\
&\quad -\int\rho\u\otimes\u:\na\mathcal{B}(h-(h)_\Omega)+\int\mathbb{S}_c:\na\mathcal{B}(h-(h)_\Omega)\\
& \leq C\Vert h\Vert_{L^{\frac{s}{s-1}}}\left(1+\Vert\left(\ln\frac{1}{\delta}\right)^{-1}\rho^{11}+\rho^2\frac{\partial f^{\de}}{\partial \rho}\Vert_{L^1}+\Vert\na\u\Vert_{L^2}+\Vert\rho|\u|^2\Vert_{L^s}+\Vert\na c\Vert_{L^{2s}}^2\right).
\end{aligned}\enn
Let us choose 
   \[h=\left(\frac{|\left(\ln\frac{1}{\delta}\right)^{-1}\rho^{11}+\rho^2\frac{\partial f^{\de}}{\partial \rho}|}{\Vert\left(\ln\frac{1}{\delta}\right)^{-1}\rho^{11}+\rho^2\frac{\partial f^{\de}}{\partial\rho}\Vert_{L^s}}\right)^{s-1},\]  then  
\be\la{ls34x}\begin{aligned}
&\Vert\left(\ln\frac{1}{\delta}\right)^{-1}\rho^{11}+\rho^2\frac{\partial f^{\de}}{\partial\rho}\Vert_{L^s}\leq C\left(1+\Vert\na\u\Vert_{L^2}+\Vert\rho|\u|^2\Vert_{L^s}+\Vert\na c\Vert_{L^{2s}}^2\right).
\end{aligned}
\ee 

We now discuss \eqref{ls34x}  in  two cases:
 \begin{enumerate}
\item 
 $\na\times \mathbf{g}_1=0.$  The energy inequality shows
\[\Vert\na\u\Vert_{L^2}^2+\Vert\na\mu\Vert_{L^2}^2\leq C.\]
Then, for all $\ga>\frac{3}{2}$,  we have the following
\bnn\ba \Vert\rho|\u|^2\Vert_{L^s}&\leq C\Vert\rho\Vert_{L^{\frac{3s}{3-s}}}\Vert\u\Vert_{L^{6}}^{2}\\
&\leq C\Vert\rho\Vert_{L^{\gamma s}}\leq C+\frac{1}{4}\Vert \rho^{\ga}\Vert_{L^s},\quad s\in (1,3-3/\ga].\ea\enn
\item $\na\times \mathbf{g}_1\neq 0$.  For all $\ga>\frac{5}{3}$,  we use    \eqref{eq2229}, \eqref{eq1.8},  and  the interpolation between $L^{1}$ and $L^{\ga s}$,   to deduce 
\bnn\ba \Vert\rho\u^2\Vert_{L^s}&\leq C\Vert\rho\Vert_{L^{\frac{3s}{3-s}}}\Vert\u\Vert_{L^{6}}^{2}\\
&\le C\Vert\rho\Vert_{L^{\frac{3s}{3-s}}}\Vert\rho\Vert_{L^{\frac{3}{2}}}\\
&\leq C\Vert\rho\Vert_{L^{\gamma s}}^{\frac{5}{3}-\frac{3\gamma-5}{3(\gamma s-1)}}\leq C+\frac{1}{4}\Vert \rho^{\ga}\Vert_{L^s},\quad s\in (1,3-3/\ga].\ea\enn
\end{enumerate}
In this connection,  under the condition \eqref{eq1.11}, it holds that
 \be\la{ls79}\ba
  C\left(\Vert\na\u \Vert_{L^{2}}+\Vert\rho |\u|^2 \Vert_{L^{s}}\right)& \leq  C  +\frac{1}{4} \Vert\rho^{\ga}\Vert_{L^{s}}\\
  &\le C+\frac{1}{4}\Vert\left(\ln\frac{1}{\delta}\right)^{-1}\rho^{11}+\rho^2\frac{\partial f^{\de}}{\partial\rho}\Vert_{L^s}.\ea\ee

It remains to bound $\|\na c\|_{L2s}^{2}.$  It follows   from $\eqref{eq3.1}_{4}$ and \eqref{eqsh15}  that
\begin{equation}\label{eqsh18}
\begin{aligned}
\Vert\Delta c\Vert_{L^{\frac{2\gamma}{1+\gamma}}}&\leq C\Vert\rho\mu\Vert_{L^{\frac{2\gamma}{1+\gamma}}}+C\Vert\rho(f_2^{\delta'}-\theta_cc)\Vert_{L^{\frac{2\gamma}{1+\gamma}}}\\
& \leq C\left(\Vert\sqrt{\rho}\mu\Vert_{L^2}+\Vert\sqrt{\rho} (f_2^{\delta'}-\theta_cc)  \Vert_{L^2}\right)\Vert\rho\Vert_{L^{\gamma}}^{\frac{1}{2}}\\
&\leq  C+C\Vert\rho\Vert_{L^{\gamma}}+C\Vert\rho\Vert_{L^{\frac{3}{2}}}^3+C\left(\ln\frac{1}{\delta}\right)^{-\frac{2}{3}}\Vert\rho\Vert_{L^4}^{\frac{52}{9}}.
\end{aligned}
\end{equation}
Using this and \eqref{eqsh1}, we perform the calculation.  
\begin{equation}\label{eqsh16}
\begin{aligned}
\Vert\na c\Vert_{L^{2s}}^2&\leq  \Vert\na c\Vert_{L^2}^{2-\frac{3\gamma(s-1)}{2\gamma s-3s}}\Vert\na c\Vert_{L^{\frac{6\gamma}{3+\gamma}}}^{\frac{6\gamma}{2\gamma-3}(1-\frac{1}{s})}\\
&\leq C\Vert\na c\Vert_{L^2}^{2-\frac{3\gamma(s-1)}{2\gamma s-3s}}\Vert\Delta c\Vert_{L^{\frac{2\gamma}{1+\gamma}}}^{\frac{6\gamma}{2\gamma-3}(1-\frac{1}{s})}\\
&\leq C\left(1+\Vert\rho\Vert_{L^{\frac{6}{5}}}^2+\left(\ln\frac{1}{\delta}\right)^{-1}\Vert\rho\Vert_{L^8}^8\right)^{1-\frac{3\gamma(s-1)}{4\gamma s-6s}}\Vert\Delta c\Vert_{L^{\frac{2\gamma}{1+\gamma}}}^{\frac{6\gamma}{2\gamma-3}(1-\frac{1}{s})}.
\end{aligned}
\end{equation}
Combining  \eqref{eqsh16} with \eqref{eqsh18},  we arrive at  
\begin{equation}\la{tty1}
\begin{aligned}
&C\Vert\na c\Vert_{L^{2s}}^2\\
&\leq  C\left(1+\Vert\rho\Vert_{L^{\frac{6}{5}}}^2+\left(\ln\frac{1}{\delta}\right)^{-1}\Vert\rho\Vert_{L^8}^8\right)\left(1+\Vert\rho\Vert_{L^{\gamma}}+\Vert\rho\Vert_{L^{\frac{3}{2}}}^3+\left(\ln\frac{1}{\delta}\right)^{-\frac{2}{3}}\Vert\rho\Vert_{L^4}^{\frac{52}{9}}\right)^{{\frac{6\gamma}{2\gamma-3}(1-\frac{1}{s})}}\\
  &\le C+\frac{1}{4}\Vert\left(\ln\frac{1}{\delta}\right)^{-1}\rho^{11}+\rho^2\frac{\partial f^{\de}}{\partial\rho}\Vert_{L^s},
\end{aligned}
\end{equation}
where,  in the last inequality,  we have chosen  $s$ close to  $1$ so that   ${\frac{6\gamma}{2\gamma-3}(1-\frac{1}{s})}$ is sufficiently small.  The desired \eqref{rho_4} thus follows from  \eqref{tty1},  \eqref{ls79},   \eqref{ls34x}. 

Therefore, in light of  \eqref{rho_4},    \eqref{eq2229},  \eqref{eqsh13}, \eqref{eqsh1}, \eqref{eqccc1},   \eqref{eqsh18}, we complete the proof of Proposition \ref{pro4}.  \end{proof}

\subsection{$\delta$-limit for the solutions}

In view of Proposition  \ref{pro4}, we are allowed to take  the  zero  $\de$-limits for solutions of \eqref{eq3.1} and \eqref{eq3.1.2}, upon  to some subsequences,
\begin{eqnarray}\label{eq4.49}
(\na\u^\delta,\na\mu^\delta)-\!\!\!\rightharpoonup(\na\u,\na\mu)\quad {\rm in}\,\,L^2,
\end{eqnarray}
\begin{eqnarray}\label{eq4.50}
(\u^\delta,\mu^\delta)\longrightarrow(\u,\mu)\quad {\rm in}\,\, L^{p_1}\,\,(p_1<6),
\quad 
c^\delta\longrightarrow c\quad {\rm in}\,\, W^{1,2},
\end{eqnarray}
\begin{eqnarray}\label{eq4.52}
\left(\ln\frac{1}{\delta}\right)^{-1}(\rho^\delta)^{11}\longrightarrow 0\quad {\rm in}\,\, L^1\quad {\rm and}\quad (\rho^\delta)^{\ga} -\!\!\!\rightharpoonup \overline{\rho^{\ga}}\quad{\rm in}\,\, L^{s}\,\,(s>1).
\end{eqnarray}
 As a result of \eqref{eq4.50} and \eqref{eq4.52},
\begin{eqnarray}\label{eq4.53}
(\rho^\delta\u^\delta,\rho^\delta\mu^\delta) -\!\!\!\rightharpoonup(\rho\u,\rho\mu)\quad{\rm in}\,\, L^{p_2}\quad({\rm for}\,\, {\rm some}\,\, p_2<6/5),
\end{eqnarray}
\begin{eqnarray}\label{eq4.54}
(\rho^\delta\u^\delta\otimes\u^\delta,\rho^\delta\u^\delta c^\delta) -\!\!\!\rightharpoonup(\rho\u\otimes\u,\rho\u c)\quad {\rm in}\,\, L^{p_3}\quad ({\rm for}\,\, {\rm some}\,\, p_3>1),
\end{eqnarray}
and furthermore,
\begin{eqnarray}\label{eq4.55}
(\rho^\delta)^2\frac{\partial f^\delta}{\partial \rho^\delta}=(\gamma-1)(\rho^\delta)^\gamma+\rho^\delta H-\!\!\!\rightharpoonup (\gamma-1)\overline{\rho^\gamma}+\rho H=\overline{\rho^2\frac{\partial f}{\partial\rho}}\quad {\rm in}\,\, L^{s},
\end{eqnarray}
\begin{eqnarray}\label{eq4.56}
\rho^\delta\frac{\partial f^\delta}{\partial c^\delta}= \rho^\delta (f_{2}^{\delta'}(c^\delta)-\theta_c c^\delta) -\!\!\!\rightharpoonup\overline{\rho\frac{\partial f}{\partial c}}\quad {\rm in}\,\, L^{{\frac{6s}{3+2s}}}.
\end{eqnarray}
Having \eqref{eq4.49}-\eqref{eq4.56} in hand, we are allowed to  take   $\delta$ limit  in \eqref{eq3.1} and obtain the following equations:
 \begin{equation}\label{eq4.57}
\left\{
\begin{aligned}
&{\rm div}(\rho\u)=0,\\
&{\rm div}(\rho\u\otimes\u)+\na\left(\overline{\rho^2\frac{\partial f}{\partial \rho}}\right)={\rm div}(\mathbb{S}_{ns}+\mathbb{S}_c)+\rho \mathbf{g}_1+\mathbf{g}_2,\\
&{\rm div}(\rho\u c)=\Delta\mu,\\
&\rho\mu=\overline{\rho\frac{\partial f}{\partial c}}-\Delta c,\\
&\int \rho =m_{1},\quad \int \rho c=m_{2}.
\end{aligned}\right.
\end{equation}
Moreover,
\[\rho\in L^{\gamma s}(\Omega),\,\,\u\in H_0^1(\Omega),\,\, \mu\in H_{n}^1(\Omega),\,\, c\in W^{2,\frac{6s}{5}}(\om),\]
and the energy inequality \eqref{ls40}  is valid.
 
\bigskip

\section{Proof of Theorem \ref{t}}

We  now prove   that the limit functions  $(\rho, \u,\mu,c)$  are  weak solutions to the  equations \eqref{eq1.1} in the sense of Definition \ref{de}. 
From \eqref{eq4.57},    it suffices to show that   \begin{eqnarray}\label{eq5.58}
\overline{\rho^2\frac{\partial f}{\partial\rho}}=\rho^2\frac{\partial f}{\partial\rho}\quad {\rm and}\quad\overline{\rho\frac{\partial f}{\partial c}}=\rho\frac{\partial f}{\partial c}\quad {\rm in}\,\,\,\,\mathcal{D}'(\om).
\end{eqnarray} 
\begin{lemma}\la{lem5.7}The density has strong convergence, i.e., 
\be\la{ls52}  \rho^{\de} \longrightarrow \rho\quad {\rm in}\,\,\,L^{1}.\ee
\end{lemma}
\begin{proof} The  proof  is available  in the Appendix.   \end{proof}
 
The first  part in \eqref{eq5.58} follows directly from Lemma \ref{lem5.7}. 
Now we are in a position to prove the second part in \eqref{eq5.58}.
It follows  from     \eqref{eqsh15} and \eqref{rho_4} that 
\be\ba\la{ls66}
&\int\rho^{\de}\left(f_{2}^{\delta '}-\theta_cc\right)^2+\int\left|f_{2}^{\delta ''}-\theta_c\right||\na c|^2\leq C,
\ea\ee
and consequently,
 \be\la{ls64}
\int_{\{x:\,c^{\de}>1\}}\rho^{\de}|f_{2}^{\delta '}-\theta_cc^{\de}|^2\leq \int_\Omega\rho^{\de}|f_{2}^{\delta '}-\theta_cc^{\de}|^2\leq C.
\ee
The definition of   $f_{2}^{\delta}$ guarantees that, for some $C=C(\theta_{0},\theta_{c})$ independent of $\de,$  
\be\la{ls65} |f_{2}^{\delta '}-\theta_cc^{\de}|^2\ge C\ln(2/\de-1)\quad {\rm in}\,\,\,\{x\in \om:\,\,|c(x)|>1\}.\ee
Thanks to  \eqref{ls64} and  \eqref{ls65}, we get 
\[\int_{\{x:|c^\delta|>1\}}\rho^\delta \d x\leq \frac{C}{\ln(2/\delta-1)}.\]
By this  we use   Fatou's lemma and deduce  \begin{eqnarray*}
\int\rho \textbf{1}_{\{x:\,|c|>1\}} \leq \lim\inf_{\delta\to 0} \int \rho^\delta \textbf{1}_{\{x:\,|c^{\de}|>1\}} \leq \lim_{\delta\to 0} \frac{C}{\ln(2/\delta-1)}= 0,
\end{eqnarray*}
which  implies  \be\la{ls60} |c|\le 1\quad a.e.\,\,\,\, {\rm in}\,\,\,\, \{x\in\Omega:\,\,\rho(x)>0\}.\ee

We further claim  that 
 \begin{eqnarray}\label{eq4.71}
-1<c(x)<1\quad  a.e.\,\,\,{\rm in}\,\,\,\{x\in \om:\,\,\rho(x)>0\}.
\end{eqnarray}
In fact,   if   $c^{\de}$ converges to $c=\pm1$, then
\be\la{ls62}\lim_{\de\rightarrow0}  f_{2}^{\delta '}(c^\delta)= \pm\infty.\ee 
While if   $c^{\de}$ converges to  some  $c\in (-1,1)$,  we see that   $$c^{\de}\in [c-\ti{\de},c+\ti{\de}] \subset [-1+\ti{\de},\,1-\ti{\de}]\subset (-1,1),$$ provided   that  $\de< \ti{\de}:=\min 
\left\{\frac{1-c}{2},\,\frac{1+c}{2}\right\}.$
  Hence, using  the definition of $f_{2}^{\delta '}(c^\delta)$ once more, we have
 \be\la{ls61} \begin{aligned} &\lim_{\de\rightarrow0} \left| f_{2}^{\delta '}(c^\delta)-f_2'(c)\right|\\
 &= \lim_{\de\rightarrow0} \left| \left( f_{2}^{\delta '}(c^\delta)-f_2'(c^{\de})\right) +\left( f_2'(c^\delta)-f_2'(c)\right)\right|\\
&=\lim_{\de\rightarrow0} \left| f_2'(c^\delta)-f_2'(c)\right|\\
&= 0,\end{aligned}\ee
where the second equality  is due to  $f_{2}^{\delta '}(c^\delta)=f_2'(c^{\de})$ as   $c^{\de}\in  [-1+\ti{\de},\,1-\ti{\de}]$.

In view of  \eqref{ls62} and  \eqref{ls61},  if  we define  
\begin{eqnarray*}
\tilde{f_{2}'}(s):=\left\{\begin{aligned}
&f_2'(s),&& {\rm if}\,\, s\in(-1,1),\\
&\pm\infty,&& {\rm if}\,\, s=\pm1,
\end{aligned}\right.
\end{eqnarray*}
 then  we deduce from  \eqref{ls52}, \eqref{ls64}, \eqref{ls62},   \eqref{ls61} that 
\bnn\ba
\int_\Omega\rho|\tilde{f_{2}'}(c)-\theta_{c} c|^2\textbf{1}_{\{x:\,\rho>0\}}&\leq\lim\inf_{\delta\to 0}\int_\Omega\rho^\delta|f_{2}^{\delta '}(c^\delta)-\theta_cc^\delta|^2\textbf{1}_{\{x:\,\rho>0\}}\\
&\leq \lim\inf_{\delta\to 0}\int_\Omega\rho^\delta|f_2^{\delta'}(c^\delta)-\theta_cc^\delta|^2 \\
&\leq  C.
\ea\enn
This proves \eqref{eq4.71}.  
\begin{remark}  By  \eqref{eq4.71},  we see that  the function $\rho (f_2'(c)-\theta_cc)$ is well defined   on the set $\{x\in \om:\,\rho(x)>0\}.$ \end{remark}

For any $\phi\in C^\infty(\overline{\Omega})$, we write 
\be\ba\label{eq4.73}
&\int\rho^\delta \left(f_{2}^{\delta '}(c^\delta)-\theta_cc^\delta\right)\phi\\
&=\int_{\{x:\,\rho=0\}}\rho^\delta  \left(f_{2}^{\delta '}(c^\delta)-\theta_cc^\delta\right)\phi +\int_{\{x:\,\rho>0\}}\rho^\delta  \left(f_{2}^{\delta '}(c^\delta)-\theta_cc^\delta\right)\phi.
\ea\ee 

On the  one hand,  from  \eqref{ls66}   we have 
\bnn\begin{aligned}
& \left|\int_{\{x:\rho=0\}}\rho^\delta  \left(f_{2}^{\delta '}(c^\delta)-\theta_cc^\delta\right)\phi\right|\\
&=\left|\int_{\{x:\,\rho=0\}}\sqrt{\rho^\delta}\sqrt{\rho^\delta} \left(f_{2}^{\delta '}(c^\delta)-\theta_cc^\delta\right)\phi\right|\\
&\leq C \Vert\sqrt{\rho^\delta}\Vert_{L^2(\{x:\,\rho=0\})}\Vert\sqrt{\rho^\delta}  \left(f_{2}^{\delta '}(c^\delta)-\theta_cc^\delta\right)\Vert_{L^2(\Omega)}\\
&\le C \Vert\sqrt{\rho^\delta}\Vert_{L^2(\{x:\,\rho=0\})},
\end{aligned}\enn
which along with the strong convergence of $\rho^{\de}$ yields
\bnn  \lim_{\de\rightarrow0}\left|\int_{\{x:\rho=0\}}\rho^\delta  \left(f_{2}^{\delta '}(c^\delta)-\theta_cc^\delta\right)\phi\right|=0.\enn
On the other hand,  by  \eqref{eq4.71}, \eqref{ls61},
 the strong convergence of $\rho^{\de}$ and $c^{\de}$,  we have 
\bnn\ba \lim_{\de\rightarrow0}\int_{\{x:\,\rho>0\}}\rho^\delta  (f_{2}^{\delta '}(c^\delta)-\theta_cc^\delta)\phi= \int_{\{x:\,\rho>0\}}\rho (f_{2}' (c)-\theta_cc)\phi .\ea\enn
Hence,  we can pass the limit in \eqref{eq4.73} and conclude 
\bnn\ba \lim_{\de\rightarrow0}\int\rho^\delta  \left(f_{2}^{\delta '}(c^\delta)-\theta_cc^\delta\right)\phi= \int\rho \left(f_{2}' (c)-\theta_cc\right)\phi ,\ea\enn
where   \begin{eqnarray*}
\rho (f_{2}' (c)-\theta_cc)=\left\{\begin{aligned}
&\rho (f_{2}' (c)-\theta_cc)&& {\rm if}\,\,\rho>0,\\
&0&&{\rm if}\,\,\rho=0.
\end{aligned}\right.
\end{eqnarray*} 
Recalling \eqref{kk}, we conclude \eqref{eq5.58}.
 The proof of Theorem \ref{t} is completed.
 
 \bigskip

\section{Appendix }

{{\bf Proof of Lemma \ref{lem3.3}.}}  Let $K$ be the Laplacian  fundamental solution   in $\mathbb{R}^3$ and 
  $\Delta^{-1}(h)=K*h$ be   the convolution of $h$.  We have   \begin{equation}\label{eq3.17}
\ba
&\Vert\partial_i\Delta^{-1}(h)\Vert_{W^{1,p}(\Omega)}\leq C(\Omega,p)\Vert h\Vert_{L^p(\mathbb{R}^3)},\quad p\in (1,\infty),\\
&\Vert\partial_i\Delta^{-1}(h)\Vert_{L^{p^*}(\Omega)}\leq C(\Omega,p)\Vert\partial_i\Delta^{-1}(h)\Vert_{W^{1,p}(\mathbb{R}^3)},\quad p^*=\frac{3p}{3-p},\,\, p<3,\\
&\Vert\partial_i\Delta^{-1}(h)\Vert_{L^\infty(\Omega)}\leq C(\Omega,p)\Vert h\Vert_{L^p(\mathbb{R}^3)},\quad p>3,
\ea
\end{equation} 
where $\partial_i\Delta^{-1}(i=1,2,3)$ is  the Mikhlin multiplier, see   \cite{Lions,Stein}.  
Assume that  $h_n\rightharpoonup h$ in $L^p(\mathbb{R}^3)$,   then
\be\ba\label{eq3.18}
\partial_j\partial_i\Delta^{-1}(h_n)\rightharpoonup\partial_j\partial_i\Delta^{-1}(h)\quad {\rm in}\,\, L^p,\quad 
\partial_i\Delta^{-1}(h_n)\to \partial_i\Delta^{-1}(h)\quad {\rm in}\,\, L^q,
\ea\ee
where $q<p^*$ if $p<3$ and $q\leq \infty$ if $p>3$.

Prolonging $\rho^\epsilon$ to  $\mathbb{R}^3$ by zero, multiplying $\eqref{eq2.7}_2^i$ by $\phi\partial_i \Delta^{-1}\rho^\epsilon$ with $\phi\in C_0^\infty(\Omega)$, we obtain
\begin{equation}\label{eq3.20}
\begin{aligned}
\int\phi\rho^\epsilon\mathbb{F}(\rho^\epsilon,\u^{\eps})=&-\int\partial_i\Delta^{-1}(\rho^\epsilon)\partial_i\phi\left(\left(\ln\frac{1}{\delta}\right)^{-1}(\rho^\epsilon)^{11}+(\rho^\epsilon)^2\frac{\partial f^\delta}{\partial\rho^\epsilon}-(\lambda_1+\lambda_2){\rm div}\u^\epsilon\right)\\
&+\lambda_1\int\left(\partial_ju_{i}^\epsilon\partial_i\Delta^{-1}(\rho^\epsilon)\partial_j\phi-u_{i}^\epsilon\partial_j\partial_i\Delta^{-1}(\rho^\epsilon)\partial_j\phi+\rho^\epsilon\u^\epsilon\cdot\na\phi\right)\\
&-\int\left(\left(\rho^\epsilon\mu^\epsilon\partial_i c^\epsilon-\rho^\epsilon\frac{\partial f^\delta}{\partial c^\epsilon}\partial_i c^\epsilon\right)\phi\partial_i\Delta^{-1}(\rho^\epsilon)-(\rho^\epsilon \mathbf{g}_1+\mathbf{g}_2)\phi\partial_i\Delta^{-1}(\rho^\epsilon)\right)\\
&-\int\rho^\epsilon u_{j}^\epsilon u_{i}^\epsilon\partial_j\phi\partial_i\Delta^{-1}(\rho^\epsilon)-\int\rho^\epsilon 
u_{j}^{\epsilon} u_{i}^{\epsilon} \phi\partial_j\partial_i\Delta^{-1}(\rho^\epsilon)\\
&+\epsilon^2\int\rho^\epsilon u_{i}^\epsilon\phi\partial_i\Delta^{-1}(\rho^\epsilon)+\epsilon^4\int\na\rho^\epsilon\cdot\na u_{i}^\epsilon \phi\partial_i\Delta^{-1}(\rho^\epsilon),
\end{aligned}
\end{equation}
where the second line on the right-hand side of \eqref{eq3.20} comes from
\begin{align*}
& \int\partial_j u_{i}^\epsilon (\partial_i\Delta^{-1}(\rho^\epsilon)\partial_j\phi+\partial_j\partial_i\Delta^{-1}(\rho^\epsilon)\phi)\\
&=  \int(\partial_j u_{i}^\epsilon \partial_i\Delta^{-1}(\rho^\epsilon)\partial_j\phi-
u_{i}^\epsilon \partial_j\partial_i\Delta^{-1}(\rho^\epsilon)\partial_j\phi-u_{i}^\epsilon \partial_i\rho^\epsilon\phi)\\
&=  \int(\partial_j u_{i}^\epsilon \partial_i\Delta^{-1}(\rho^\epsilon)\partial_j\phi-
u_{i}^\epsilon\partial_j\partial_i\Delta^{-1}(\rho^\epsilon)\partial_j\phi+\rho^\epsilon \u^\epsilon\cdot\na\phi)+ \int\rho^\epsilon{\rm div}\u^\epsilon\phi.
\end{align*}
By the  regularity $(\rho^\epsilon,\u^\epsilon)\in(H^2,H_0^1)$ and  $\frac{\partial\rho^\epsilon}{\partial \vec{n}}|_{\partial \Omega}=0$, we compute 
$$-\int\rho^\epsilon u_{i}^\epsilon \phi\partial_i\partial_j\Delta^{-1}(\rho^\epsilon u_{j}^\epsilon)   = -\epsilon^4\int\rho^\epsilon u_{i}^\epsilon\phi\partial_i\Delta^{-1}({\rm div}(\textbf{1}_\Omega\na\rho^\epsilon))+\epsilon^2\int\rho^\epsilon u_{i}^\epsilon\phi\partial_i\Delta^{-1}(\rho^\epsilon-\rho_0),$$
and
\begin{equation}\label{eq3.21}
\begin{aligned}
&-\int\rho^\epsilon u^\epsilon_j u^\epsilon_i\partial_j\phi\partial_i\Delta^{-1}(\rho^\epsilon)-\int\rho^\epsilon u^\epsilon_j u^\epsilon_i\phi\partial_j\partial_i\Delta^{-1}(\rho^\epsilon)\\
& =-\int\rho^\epsilon u^\epsilon_j u^\epsilon_i\partial_j\phi\partial_i\Delta^{-1}(\rho^\epsilon)+\int u^\epsilon_i\phi[\rho^\epsilon\partial_i\partial_j\Delta^{-1}(\rho^\epsilon u^\epsilon_j)-\rho^\epsilon u^\epsilon_j\partial_j\partial_i\Delta^{-1}(\rho^\epsilon)]\\
&\quad -\int\rho^\epsilon u^\epsilon_i\phi\partial_i\partial_j\Delta^{-1}(\rho^\epsilon u^\epsilon_j)\\
&  = -\int\rho^\epsilon u^\epsilon_j u^\epsilon_i\partial_j\phi\partial_i\Delta^{-1}(\rho^\epsilon)+\int u^\epsilon_i\phi[\rho^\epsilon\partial_i\partial_j\Delta^{-1}(\rho^\epsilon u^\epsilon_j)-\rho^\epsilon u^\epsilon_j\partial_j\partial_i\Delta^{-1}(\rho^\epsilon)]\\
& \quad- \epsilon^4\int\rho^\epsilon u^\epsilon_i\phi\partial_i\Delta^{-1}({\rm div}(\textbf{1}_\Omega\na\rho^\epsilon))+\epsilon^2\int\rho^\epsilon u^\epsilon_i\phi\partial_i\Delta^{-1}(\rho^\epsilon-\rho_0).
\end{aligned}
\end{equation}
Replace the second line from the bottom in \eqref{eq3.20} by \eqref{eq3.21}  to find 
\begin{align}\label{eq3.22}
&\int\phi\rho^\epsilon\mathbb{F}(\rho^\epsilon,\u^{\eps})\notag\\
=&-\int\partial_i\Delta^{-1}(\rho^\epsilon)\partial_i\phi\left(\left(\ln\frac{1}{\delta}\right)^{-1}(\rho^\epsilon)^{11}+(\rho^\epsilon)^2\frac{\partial f^\delta}{\partial\rho^\epsilon}-(\lambda_1+\lambda_2){\rm div}\u^\epsilon\right)\notag\\
&+\lambda_1\int\left(\partial_ju_{i}^\epsilon\partial_i\Delta^{-1}(\rho^\epsilon)\partial_j\phi-u_{i}^\epsilon\partial_j\partial_i\Delta^{-1}(\rho^\epsilon)\partial_j\phi+\rho^\epsilon\u^\epsilon\cdot\na\phi\right)\notag\\
&-\int\left(\left(\rho^\epsilon\mu^\epsilon\partial_i c^\epsilon-\rho^\epsilon\frac{\partial f^\delta}{\partial c^\epsilon}\partial_i c^\epsilon\right)\phi\partial_i\Delta^{-1}(\rho^\epsilon)-(\rho^\epsilon \mathbf{g}_1+\mathbf{g}_2)\phi\partial_i\Delta^{-1}(\rho^\epsilon)\right)\notag\\
&-\int\rho^\epsilon u^\epsilon_j u^\epsilon_i\partial_j\phi\partial_i\Delta^{-1}(\rho^\epsilon)+\int u^\epsilon_i\phi[\rho^\epsilon\partial_i\partial_j\Delta^{-1}(\rho^\epsilon u^\epsilon_j)-\rho^\epsilon u^\epsilon_j\partial_j\partial_i\Delta^{-1}(\rho^\epsilon)]\notag\\
&- \epsilon^4\int\rho^\epsilon u^\epsilon_i\phi\partial_i\Delta^{-1}({\rm div}(\textbf{1}_\Omega\na\rho^\epsilon))-\na\rho^\epsilon\cdot\na u^\epsilon_i\phi\partial_i\Delta^{-1}(\rho^\epsilon)+\epsilon^2\int\rho^\epsilon u^\epsilon_i\phi\partial_i\Delta^{-1}(2\rho^\epsilon-\rho_0)\notag\\
=&\sum_{i=1}^7\mathcal{J}_i^\epsilon.
\end{align}
On the other hand, if we take $\epsilon$-limit in $\eqref{eq2.7}_2$  and then  multiply  it  by $\phi\partial_i\Delta^{-1}(\rho)$, we obtain
\begin{align}\label{eq3.23}
\int\phi\rho\overline{\mathbb{F}(\rho,\u)}=&-\int\partial_i\Delta^{-1}(\rho)\partial_i\phi\left(\left(\ln\frac{1}{\delta}\right)^{-1}\rho^{11}+\overline{\rho^2\frac{\partial f^\delta}{\partial \rho}}-(\lambda_1+\lambda_2){\rm div}\u\right)\notag\\
&+\lambda_1\int(\partial_j u^i\partial_i\Delta^{-1}(\rho)\partial_j\phi-u^i\partial_j\partial_i\Delta^{-1}(\rho)\partial_j\phi+\rho\u\cdot\na\phi)\notag\\
&-\int\left(\left(\rho\mu\partial_i c+\overline{\rho\frac{\partial f^\delta}{\partial c}}\partial_i c\right)\phi\partial_i\Delta^{-1}(\rho)-(\rho \mathbf{g}_1+\mathbf{g}_2) \phi\partial_i\Delta^{-1}(\rho)\right)\notag\\
&-\int\rho u^j u^i\partial_j\phi\partial_i\Delta^{-1}(\rho)+\int u^i\phi[\rho\partial_i\partial_j\Delta^{-1}(\rho u^j)-\rho u^j\partial_j\partial_i\Delta^{-1}(\rho)]\notag\\
=&\sum_{i=1}^5\mathcal{J}_i.
\end{align}
In terms of \eqref{eq3.22} and \eqref{eq3.23}, in order to prove \eqref{eq3.16}  we need justify 
\[\lim_{\epsilon\to 0}\mathcal{J}_i^\epsilon=\mathcal{J}_i\,\,(i=1\sim 5)\quad {\rm and}\quad \lim_{\epsilon\to 0}\mathcal{J}_i^\epsilon=0\,\,(i=6,7).\]
Making use of \eqref{eq3.18},   \eqref{eq3.3}-\eqref{eq3.9}, and the strong convergence of $\na c^{\eps}$,  we easily check  $\lim_{\epsilon\to 0}\mathcal{J}_i^\epsilon=\mathcal{J}_i$  for $i=2,3,4$. 
Due to the uniform estimates in  Lemma \ref{lem3.1} and  \eqref{eq2.224}-\eqref{eq2.246},   we see that 
\bnn\ba
|\mathcal{J}_6^\epsilon+\mathcal{J}_7^\epsilon|&\leq\epsilon^4\Vert\na\rho^\epsilon\Vert_{L^2}\Vert\rho^\epsilon\Vert_{L^3}\Vert\u^\epsilon\Vert_{L^6}+\epsilon^2\Vert\rho^\epsilon\Vert_{L^2}\Vert\u^\epsilon\Vert_{L^2}\Vert\partial_i\Delta^{-1}(2\rho^\epsilon-\rho_0)\Vert_{L^\infty}\\
&+\epsilon^4\Vert\na\rho^\epsilon\Vert_{L^2}\Vert\na\u^\epsilon\Vert_{L^2}\Vert\partial_i\Delta^{-1}(\rho^\epsilon)\Vert_{L^\infty}\\ 
&\leq C\epsilon\to 0\quad {\rm as}\,\,\,\,\eps\to 0.
\ea\enn
Finally,    $\lim_{\epsilon\to 0}\mathcal{J}_5^\epsilon=\mathcal{J}_5$ is  from  the following lemma,  by choosing  $v^\epsilon=\rho^\epsilon u^\epsilon_j$ and $w^\epsilon=\rho^\epsilon$.

\begin{lemma}(\!\!\cite[Section 5]{Evans2}) Let $r_{1}^{-1}+r_{2}^{-1}=r^{-1}\le 1$ and 
$$v^{\eps} -\!\!\!\rightharpoonup v \,\,\,{\rm in}\,\,\,\,L^{r^{1}}\quad {\rm and}\quad w^{\eps} -\!\!\!\rightharpoonup w\,\,\,{\rm in}\,\,\,\,L^{r^{2}}.$$
Then,
$$v^{\eps} \p_{i}\p_{j}\lap^{-1}(w^{\eps})-w^{\eps} \p_{i}\p_{j}\lap^{-1}(v^{\eps})=v \p_{i}\p_{j}\lap^{-1}(w)-w \p_{i}\p_{j}\lap^{-1}(v)\,\,\,{\rm in}\,\,\,\,L^{r}.$$
\end{lemma}
 
 \bigskip
 
{ {\bf Proof of Lemma \ref{lem5.7}.}}
The proof  is as in  \cite{FNP} and the  main idea is to use the cutoff technique to guarantee that $(\rho,\u)$ is still a renormalized solution, so that the defection of the density can be controlled.

Define   \begin{equation}\label{eq4.58}
C^1[0,\infty)\ni T_k(z)=\left\{\begin{aligned}
&z,&& z\leq k,\\
&k+0.5,&& z\geq k+1.
\end{aligned}\right.
\end{equation}
We claim that 
\begin{equation}\label{eq4.60}
\begin{aligned}
\lim_{\delta\to 0}\int T_k(\rho^\delta)\left((\rho^\delta)^2\frac{\partial f^\delta}{\partial \rho^\delta}-(2\lambda_1+\lambda_2){\rm div}\u^\delta\right) =\int\overline{T_k(\rho)}\left(\overline{\rho^2\frac{\partial f}{\partial\rho}}-(2\lambda_1+\lambda_2){\rm div}\u\right).
\end{aligned}
\end{equation} 
In fact,  multiply   $\eqref{eq3.1}_{2}$  by $\Phi=\phi\na\Delta^{-1}(T_k(\rho^\delta))$ and    $\eqref{eq4.57}_2$  by  $\phi\na\Delta^{-1}(\overline{T_k(\rho)})$ respectively,  to deduce 
\bnn\begin{aligned}
&\int\phi T_k(\rho^\delta)\left((\rho^\delta)^2\frac{\partial f^\delta}{\partial\rho^\delta}-(2\lambda_1+\lambda_2){\rm div}\u^\delta\right)\\
& =-\int\left(\ln\frac{1}{\delta}\right)^{-1}\phi T_k(\rho^\delta)(\rho^\delta)^{11}\\
&\quad -\int\partial_i\Delta^{-1}(T_k(\rho^\delta))\partial_i\phi\left(\left(\ln\frac{1}{\delta}\right)^{-1}(\rho^\delta)^{11}+(\rho^\delta)^2\frac{\partial f^\delta}{\partial\rho^\delta}-(\lambda_1+\lambda_2){\rm div}\u^\delta\right)\\
&\quad  +\lambda_1\int\partial_j u^\delta_i\partial_i\Delta^{-1}(T_k(\rho^\delta))\partial_j\phi-u^\delta_i\partial_j\partial_i\Delta^{-1}(T_k(\rho^\delta))\partial_j\phi+T_k(\rho^\delta)\u^\delta\cdot\na\phi\\
&\quad  -\int(\rho^\delta \mathbf{g}_1+\mathbf{g}_2 )^i\phi\partial_i\Delta^{-1}(T_k(\rho^\delta))  +\frac{1}{2}\int|\na c^\delta|^2(\phi T_k(\rho^\delta)+\partial_i\phi\partial_i\Delta^{-1}(T_k(\rho^\delta)))\\
&\quad  -\int\partial_i c^\delta\partial_j c^\delta(\phi\partial_j\partial_i\Delta^{-1}(T_k(\rho^\delta)) +\partial_j\phi\partial_i\Delta^{-1}(T_k(\rho^\delta))) -\int\rho^\delta u^\delta_j u^\delta_i\partial_j\phi\partial_i\Delta^{-1}(T_k(\rho^\delta))\\
&\quad  -\int u^\delta_i\phi[\rho^\delta u^\delta_j\phi\partial_j\partial_i\Delta^{-1}(T_k(\rho^\delta))-T_k(\rho^\delta)\partial_i\partial_j\Delta^{-1}(\rho^\delta u^\delta_j)]\\
& =\sum_{i=1}^7R_i^\delta
\end{aligned}\enn
and 
\begin{align*}
&\int\phi\overline{T_k(\rho)}\left(\overline{\rho^2\frac{\partial f}{\partial \rho}}-(2\lambda_1+\lambda_2){\rm div}\u\right)\\
& =-\int\partial_i\Delta^{-1}(\overline{T_k(\rho)})\partial_i\phi\left(\overline{\rho^2\frac{\partial f}{\partial \rho}}-(\lambda_1+\lambda_2){\rm div}\u\right)\\
&\quad +\lambda_1\int\partial_j u^i\partial_i\Delta^{-1}(\overline{T_k(\rho)})\partial_j\phi- u^i\partial_j\partial_i\Delta^{-1}(\overline{T_k(\rho)})\partial_j\phi+\overline{T_k(\rho)}\u\cdot\na\phi\\
&\quad  -\int(\rho \mathbf{g}_1+\mathbf{g}_2 )\phi\partial_i\Delta^{-1}(\overline{T_k(\rho)}) +\frac{1}{2}\int|\na c|^2(\phi\overline{T_k(\rho)}+\partial_i\phi\partial_i\Delta^{-1}(\overline{T_k(\rho)}))\\
&\quad  -\int\partial_i c\partial_j c(\phi\partial_j\partial_i\Delta^{-1}(\overline{T_k(\rho)})+\partial_j\phi\partial_i\Delta^{-1}(\overline{T_k(\rho)})  -\int\rho u^j u^i\partial_j\phi\partial_i\Delta^{-1}(\overline{T_k(\rho)})\\
& \quad -\int u^i\phi[\rho u^j\partial_j\partial_i\Delta^{-1}(\overline{T_k(\rho)})-\overline{T_k(\rho)}\partial_i\partial_j\Delta^{-1}(\rho u^j)]\\
& =\sum_{i=1}^7R_i.
\end{align*}
In this connection, we use  the same argument  as  in   Lemma \ref{lem3.3} and conclude 
\bnn
\lim_{\delta\to 0}R_i^\delta=R_i\quad(i=1\sim 7),
\enn
 and therefore,  the desired  \eqref{eq4.60}   is obtained. 

Having   \eqref{eq4.60}  obtained,  we compute
\begin{equation}\label{eq4.63x}
\ba
&(2\lambda_1+\lambda_2)\lim_{\delta\to 0}\int\left(T_k(\rho^\delta){\rm div}\u^\delta-\overline{T_k(\rho)}{\rm div}\u\right)\\
& =\lim_{\delta\to 0}\int\left(T_k(\rho^\delta)(\rho^\delta)^2\frac{\partial f^\de}{\partial\rho^\delta}-\overline{T_k(\rho)}\overline{\rho^2\frac{\partial f}{\partial \rho}}\right)\\
& =\lim_{\delta\to 0}\int\left((\rho^\delta)^2\frac{\partial f^{\de}}{\partial \rho^\delta}-\rho^2\frac{\partial f(\rho,c)}{\partial\rho}\right)\left(T_k(\rho^\delta)-T_k(\rho)\right)\\
&\quad  +\int\left(\overline{\rho^2\frac{\partial f}{\partial \rho}}-\rho^2\frac{\partial f(\rho,c)}{\partial \rho}\right)\left(T_k(\rho^\delta)-T_k(\rho)\right)\\
& \geq\lim_{\delta\to 0}\int\left((\rho^\delta)^2\frac{\partial f^\delta}{\partial\rho^\delta}-\rho^2\frac{\partial f(\rho,c)}{\partial \rho}\right)\left(T_k(\rho^\delta)-T_k(\rho)\right),
\ea
\end{equation}
where the last inequality sign  is due to the concavity of $T_k$ and
\[\overline{\rho^2\frac{\partial f}{\partial \rho}}=(\gamma-1)\overline{\rho^\gamma}+\overline{\rho\ln \rho}H\geq (\gamma-1)\rho^\gamma+\rho\ln\rho H=\rho^2\frac{\partial f(\rho,c)}{\partial\rho}.\]
By the fact 
$((\rho^\delta)^\gamma-\rho^\gamma)(T_k(\rho^\delta)-T_k(\rho))\geq (T_k(\rho^\delta)-T_k(\rho))^{\gamma+1},$
 the last integral in \eqref{eq4.63x} satisfies 
\begin{align*}
&\int\left((\rho^\delta)^2\frac{\partial f^\delta}{\partial\rho^\delta}-\rho^2\frac{\partial f(\rho, c)}{\partial\rho}\right)\left(T_k(\rho^\delta)-T_k(\rho)\right)\\
& =\int(\gamma-1)((\rho^\delta)^\gamma-\rho^\gamma)(T_k(\rho^\delta)-T_k(\rho))+(\rho^\delta-\rho)H\left(T_k(\rho^\delta)-T_k(\rho)\right)\\
& \geq\int(\gamma-1)(T_k(\rho^\delta)-T_k(\rho))^{\gamma+1} +(\rho^\delta-\rho)H\left(T_k(\rho^\delta)-T_k(\rho)\right).
\end{align*}
This  and  \eqref{eq4.63x}  provide  us 
\begin{equation}\label{eq4.63}
\begin{aligned}
&(2\lambda_1+\lambda_2)\lim_{\delta\to 0}\int(T_k(\rho^\delta){\rm div}\u^\delta-\overline{T_k(\rho)}{\rm div}\u)\geq (\gamma-1)\lim_{\delta\to 0 }\int(T_k(\rho^\delta)-T_k(\rho))^{\gamma+1}.
\end{aligned}
\end{equation}

Introduce 
\[b_k(z):=L_k(z)-\left(\ln k+\int_k^{k+1}\frac{T_k(s)}{s^2}ds+1\right)z,\]
where \begin{eqnarray*}
L_k(z):=\left\{\begin{aligned}
& z\ln z,&&z\leq k,\\
&z\ln k+z\int_k^z\frac{T_k(s)}{s^2}ds,&&z\geq k.
\end{aligned}\right.
\end{eqnarray*}
We see that  $b_k(z)\in C[0,\infty)\cap C^1(0,\infty)$, $b_k'(z)=0$ if $z\geq k+1$, and $b_k'(z)z-b_k(z)=T_k(z)$. Then, from the mass equation we obtain  (approximating $b_k(z)$ near $z=0$)
${\rm div}(b_k(\rho^\delta)\u^\delta)+T_k(\rho^\delta){\rm div}\u^\delta=0$.
Consequently, 
\bnn
\int T_k(\rho^\delta){\rm div}\u^\delta=0.
\enn
Recalling the  definition of $T_{k}$,  we see that,   for any fixed   $k\in\mathbb{N}$, 
\bnn
T_k(\rho^\delta)-\!\!\!\rightharpoonup\overline{T_k(\rho)}\quad {\rm in}\,\, L^p(\Omega),\quad\forall p\in[1,\infty].
\enn
Hence,
\begin{eqnarray}\label{eq4.65}
\int\overline{T_k(\rho)}\, {\rm div}\u=\lim_{\delta\to 0}\int T_k(\rho^\delta){\rm div}\u^\delta=0.
\end{eqnarray}
By the H\"older inequality,  \eqref{eq4.63} and \eqref{eq4.65}, we have 
\bnn\ba
C\Vert T_k(\rho)-\overline{T_k(\rho)}\Vert_{L^2}&\geq(2\lambda_1+\lambda_2)\int\left(T_k(\rho)-\overline{T_k(\rho)}\right){\rm div}\u\\
&=(2\lambda_1+\lambda_2)\lim_{\delta\to 0}\int\left(T_k(\rho^\delta){\rm div}\u^\delta-\overline{T_k(\rho)}{\rm div}\u\right)\\
&\geq(\gamma-1)\lim_{\delta\to 0}\int(T_k(\rho^\delta)-T_k(\rho))^{\gamma+1},
\ea\enn
and therefore,  
\begin{equation}\label{eq4.67}
\begin{aligned}
\lim_{k\to \infty}\lim_{\delta\to 0}\Vert T_k(\rho)-T_k(\rho^\delta)\Vert_{L^{\gamma+1}}^{\gamma+1}&\leq C\lim_{k\to \infty}\Vert T_k(\rho)-\overline{T_k(\rho)}\Vert_{L^2}\\
&\leq C\lim_{k\to \infty}\lim_{\delta\to 0}(\Vert T_k(\rho)-\rho\Vert_{L^2}+\Vert T_k(\rho^\delta)-\rho^\delta\Vert_{L^2}).
\end{aligned}
\end{equation}
However, by the uniform bound 
$\Vert\rho^\delta\Vert_{L^{2}}^{2}\leq  \Vert\rho^\delta\frac{\partial f^{\de}}{\partial\rho^\delta}\Vert_{L^s}^s\leq C$,
 one has
\begin{eqnarray*}
\Vert T_k(\rho^\delta)-\rho^\delta\Vert_{L^{2}}=\Vert T_k(\rho^\delta)-\rho^\delta\Vert_{L^{2}({\rho^\delta\geq k})}\leq 2\Vert\rho^\delta\Vert_{L^{2}(\rho^\delta\geq k)}\to 0,\quad {\rm as}\,\,k\to \infty.
\end{eqnarray*} Thus, 
\begin{eqnarray}\label{eq4.68}
\lim_{k\to \infty}\lim_{\delta\to 0}(\Vert T_k(\rho)-\rho\Vert_{L^2}+\Vert T_k(\rho^\delta)-\rho^\delta\Vert_{L^2})=0.
\end{eqnarray} 
The combination of   \eqref{eq4.67} with  \eqref{eq4.68} guarantees 
\begin{align*}
&\lim_{\delta\to 0}\Vert\rho^\delta-\rho\Vert_{L^1}\\
&\leq \lim_{k\to \infty}\lim_{\delta\to 0}\left(\Vert\rho^\delta-T_k(\rho^\delta)\Vert_{L^1}+\Vert T_k(\rho^\delta)-T_k(\rho)\Vert_{L^1}+\Vert T_k(\rho)-\rho\Vert_{L^1}\right)\\
&=0.
\end{align*}

\section*{Acknowledgments.} 
The research of Z. Liang was  partially  supported by the NNSFC of China grant 12531010 and the Sichuan
Science and Technology Program grant 2025ZYD0157.
J. Shuai   was partially  supported by Civil Aviation Flight University of China Program grant 25CAFUC04076.  
D. Wang was partially  supported in part by NSF grant  DMS-2510532.

\end{document}